\documentclass[english]{amsart}

\usepackage{esint}
\usepackage[svgnames]{xcolor} 
\usepackage[colorlinks,citecolor=red,pagebackref,hypertexnames=false,breaklinks]{hyperref}
\usepackage{pgf,tikz}
\usepackage{pdfsync}

\usepackage{dsfont}
\usepackage{url}
\usepackage[utf8]{inputenc}
\usepackage[T1]{fontenc}
\usepackage{lmodern}
\usepackage{babel}
\usepackage{mathtools}  
\usepackage{amssymb}
\usepackage{lipsum}
\usepackage{mathrsfs}
\usepackage{color}

\newtheorem{theorem}{Theorem}[section]
\newtheorem{proposition}{Proposition}[section]
\newtheorem{lemma}{Lemma}[section]

\newtheorem{corollary}{Corollary}[section]

\newtheorem{remark}{Remark}[section]

\numberwithin{equation}{section}

\title[Parabolic Cauchy problems]{Global stability result for 
parabolic Cauchy problems}

\thanks{MC is supported by the grant ANR-17-CE40-0029 of the French National 
Research Agency ANR (project MultiOnde). MY is partially supported by 
Grant-in-Aid for Scientific Research (S) 15H05740 of Japan Society for the 
Promotion of Science, and this article was 
prepared with the support of the "RUDN University Program 5-100".
}
 
\author[Mourad Choulli]{Mourad Choulli}
\address{Universit\'e de Lorraine, 34 cours L\'eopold, 54052 Nancy cedex, 
France}
\email{mourad.choulli@univ-lorraine.fr}

\author[Masahiro Yamamoto]{Masahiro Yamamoto}
\address{Department of Mathematical Sciences, The University of Tokyo 3-8-1, 
Komaba, Meguro, Tokyo 153, Japan,
\newline
Honorary Member of Academy of Romanian Scientists,  
Splaiul Independentei Street, no 54,
050094 Bucharest Romania,
\newline
Peoples' Friendship University of Russia (RUDN University),
6 Miklukho-Maklaya St, Moscow, 117198, Russian Federation
}
\email{myama@ms.u-tokyo.ac.jp}

\begin{document}

\begin{abstract}
Uniqueness of parabolic Cauchy problems is nowadays a classical problem and
since Hadamard \cite{Ha} these kind of problems are known to be ill-posed and 
even severely ill-posed. Until now there are only few partial results 
concerning 
the quantification of the stability of parabolic Cauchy problems. 
We bring in the present work an answer to this issue for smooth solutions  
under the minimal condition that the domain is Lipschitz.
\end{abstract}

\subjclass[2010]{35R25, 35K99, 58J35}

\keywords{Parabolic Cauchy problems, logarithmic stability, Carleman inequality, Hardy inequality.}

\maketitle

\tableofcontents

\section{Introduction}\label{section1}

Throughout this article 
$\Omega$ is a bounded domain of $\mathbb{R}^n$ with Lipschitz boundary 
$\Gamma$. Consider the parabolic operator
\[
L=\mbox{div}(A\nabla \, \cdot )-\partial _t.
\]
Here $A=(a^{ij})$ is a symmetric matrix  whose elements belong 
to $W^{1,\infty}(\Omega )$.  Furthermore assume that there exists a constant 
$0<\kappa \le 1$ so that 
\begin{equation}\label{0.1}
A(x)\xi \cdot \xi \geq \kappa |\xi |^2,\quad  x\in \Omega , \; \xi \in 
\mathbb{R}^n,
\end{equation}
and 
\begin{equation}\label{0.2}
\|a^{ij}\|_{W^{1,\infty}(\Omega )}\le \kappa^{-1},\quad 1\leq i,j\leq n.
\end{equation}

Let $t_0<t_1$ satisfy $t_1-t_0\le T_0$, for some given $T_0>0$ and 
let us set $Q=\Omega \times (t_0,t_1)$. 

Recall the notation
\[
H^{2,1}(Q)=L^2((t_0,t_1),H^2(\Omega ))\cap H^1((t_0,t_1),L^2(\Omega )).
\]

If $\Gamma_0$ is a nonempty open subset of $\Gamma$, then a classical result 
says that any $u\in H^{2,1}(Q)$ satisfying $Lu=0$ in $Q$ and $u=\nabla u=0$ on 
$\Gamma_0\times (t_0,t_1)$ must be identically equal to zero in $Q$ (see e.g., 
\cite{Chou} and the references therein). This result is known as the 
uniqueness of the Cauchy problem for the equation $Lu=0$. Quantifying this 
uniqueness result consists in controlling a norm of a solution of $Lu=0$ 
by a suitable function of the norm of 
$(u,\nabla u)_{|\Gamma_0\times (t_0,t_1)}$ in some  space.

In the rest of this paper, $0<\alpha <1$ is fixed and, for the sake of 
simplicity, we use the notations
\begin{align*}
&\mathscr{X}(Q)=C^{1+\alpha ,(1+\alpha )/2}(\overline{Q})\cap H^1((t_0,t_1),
H^2(\Omega )),
\\
&\mathscr{Y}(Q)=\{u\in \mathscr{X}(Q);\; \partial_tu\in \mathscr{X}(Q)\},
\\
&\mathscr{Z}(Q)=\mathscr{Y}(Q)\cap H^3((t_0,t_1),H^2(\Omega )).
\end{align*}
We endow $\mathscr{X}(Q)$, $\mathscr{Y}(Q)$ and $\mathscr{Z}(Q)$  with their 
natural norms
\begin{align*}
&\|u\|_{\mathscr{X}(Q)}=\|u\|_{C^{1+\alpha ,(1+\alpha )/2}(\overline{Q})}
+ \|u\|_{H^1((t_0,t_1),H^2(\Omega ))}\\
&\|u\|_{\mathscr{Y}(Q)}=\|u\|_{\mathscr{X}(Q)}+\|\partial_tu\|_{\mathscr{X}
(Q)}. \\
& \|u\|_{\mathscr{Z}(Q)}=\|u\|_{\mathscr{Y}(Q)}+\|u\|_{H^3((t_0,t_1),
H^2(\Omega ))}
\end{align*}

Let $\varrho^\ast =e^{-e^{e}}$. Set then, for $\mu >0$ and $\varrho_0\le \varrho^\ast$,
\[
\Psi_{\varrho_0,\mu}(\varrho )=
\left\{
\begin{array}{ll}
0 &\mbox{if}\; \varrho =0,
\\
(\ln \ln  \ln  |\ln \varrho |)^{-\mu} \quad &\mbox{if}\; 0<\varrho\le 
\varrho_0 ,
\\
\varrho &\mbox{if}\; \varrho \ge \varrho_0.
\end{array}
\right.
\]

We are mainly concerned in the present work with the stability issue for the 
Cauchy problem associated to the parabolic operator $L$. Precisely, we are 
going to prove the following result.

\begin{theorem}\label{theorem1}
We assume that $Lu = 0$ in $Q$.
Let $\Gamma_0$ be a nonempty open subset of $\Gamma$ and $s\in (0,1/2)$. 
Then there exist two constants $C>0$ and $0<\varrho_0\le \varrho^\ast$, 
depending on $\Omega$, $\kappa$, $T_0$, $\alpha$, $s$ and $\Gamma_0$, 
so that, for any $0\ne u\in\mathscr{Z}(Q)$ satisfying $Lu=0$ in $Q$, we have
\[
C\|u\|_{L^2((t_0,t_1),H^1(\Omega ))}\le \|u\|_{\mathscr{Z}(Q)}
\Psi_{\varrho_0,\mu} \left(\frac{\mathcal{C}(u,\Gamma _0)}{\|u\|
_{\mathscr{Z}(Q)}}\right),
\]
with $\mu =\min\{\alpha,s\}/4$ and
\[
\mathcal{C}(u,\Gamma _0)=\|u\|_{H^3((t_0,t_1),L^2(\Gamma _0))}+\|\nabla u\|_{H^2((t_0,t_1),L^2(\Gamma _0))}.
\]
\end{theorem}

We note that Theorem 1.1 asserts a global estimate in the whole domain $Q$ 
where the equation $Lu=0$ holds. 

It is straightforward to check that $\mathcal{C}(u,\Gamma _0)$ in the 
preceding theorem can be substituted by
\[
\mathcal{C}(u,\Gamma _0)=\|u\|_{H^3((t_0,t_1),L^2(\Gamma _0))\cap 
H^2((t_0,t_1),H^1(\Gamma _0))}+\|\partial_{\mathbf{n}} u\|
_{H^2((t_0,t_1),L^2(\Gamma _0))}.
\]
Here $\mathbf{n}$ is the unit exterior normal field to $\Gamma$ and 
$\partial_{\mathbf{n}} u=\nabla u\cdot \mathbf{n}$.

It has been already well-known that one can prove estimates of
H\"older type in proper subdomains of $\Omega \times 
(t_0,t_1)$.  As for references, see \cite{Ya} for example and
in particular, see \cite[Theorem 5.1, page 24]{Ya} 
where a H\"older stability estimate by Cauchy data on $\Gamma_0
\times (t_0,t_1)$ is proved in a domain 
$\Omega' \times (\tilde{t_0}, \tilde{t_1})$, where
$\Gamma_0 \subset \partial\Omega$ is an arbitrarily given subboundary, 
$\tilde{t_0}$ and $\tilde{t_1}$ are arbitrarily given 
such that $t_0 < \tilde{t_0} < \tilde{t_1} < t_1$ and
$\Omega'$ is an arbitrary subdomain such that 
$\Omega' \Subset \Omega$.
Moreover it is known that when one wants to 
estimate $u$ up to boundary points, the stability 
rate becomes worse such as logarithmic, but to the best knowledge of the 
authors, the global estimate in the whole
domain $\Omega \times (t_0,t_1)$ has been not known by data on 
an arbitrary small subboundary $\Gamma_0$, and our theorem is the first
result.

Recently, for $L=\Delta -\partial _t$,
Bourgeois \cite[Main theorem, page 2]{Bo} proved a result 
similar to the one in Theorem \ref{theorem1}, which is of a single logarithmic 
type, but for such better rate than ours, he assumes a special geometrical 
configuration for $\Gamma_0$.  More precisely, he considers 
the case where $\Omega =D\setminus O$, $D$ and $O$ are two domains 
of class $C^2$, $O\Subset D$, and $\Gamma _0$ is either $\partial D$ or 
$\partial O$, and in particular, $\Gamma_0$ cannot be taken arbitrarily.
His result is based on a Carleman estimate with a weight 
function built from the distance to the boundary of the space variable,
which is different from our Carleman estimate.

Unlike \cite{Bo}, we discuss in the present work the Cauchy problem in all of 
its generality, that is without any restriction on the part of the boundary 
where the Cauchy data are given.

We also mention that Vessella \cite{Ve} established a local H\"older stability estimate, for a parabolic Cauchy problem, up to $t=t_1$ but far from $t=t_0$. Although the tools used in \cite{Ve} are similar to ours his approach is completely different from the one we use in the present work. So it is not clear whether it is possible combine our results with those in \cite{Ve} to improve our results.

We observe that Theorem \ref{theorem1} remains valid if $L$ is substituted by 
$L$ plus an operator of first order in space variable whose coefficients are 
bounded.
We limited ourselves to the case 
$Lu=0$, but one can consider the case of 
non-zero right-hand side of $Lu$ by adding to 
$\mathcal{C}(u,\Gamma _0)$ the norm of $Lu$ in a suitable space. We believe that this extension can be done without major difficulties.

The existing stability inequalities for elliptic Cauchy problems are optimal for $H^2$-solutions and $C^{1,1}$-domains \cite{Bo} and $C^{1,1}$-solutions and Lipschitz domains \cite{Ch}, \cite[Appendix A]{BC} or \cite{Ch1}. A non optimal result was recently proved by the first author \cite{Choulli} for $H^2$ solutions and Lipschitz domains.

The proof of the main result is inspired by that used in the elliptic case by 
the first author in \cite{Ch}. Note however that there is a great difference between 
the elliptic case and the parabolic case. 

In the elliptic case the main tool is a three-ball inequality with arbitrary 
radius. While in the parabolic case the method is based on a three-cylinder 
inequality (see Theorem \ref{theorem2.1}) in which the radius depends on the 
distance to the boundary of the time variable. Roughly speaking, the radius 
becomes smaller and smaller as the time variable approaches the boundary. 

Let us explain briefly how we obtain the global stability in the elliptic case. It consists in three steps. The first step consists in a continuation argument from an interior subdomain to the boundary which gives a stability inequality of logarithmic type. In the second step we continue solutions from an interior subdomain to another interior subdomain. While in the third step we continue the Cauchy data to an interior subdomain. These last two steps lead to H\"older stability inequalities. We follow the same scheme to establish the global stability inequality for parabolic equations.  The fact that we use in the first two steps a non optimal three-cylinder inequality yields a logarithm in the final inequality. We use in the third step directly the Carleman inequality but here again the absence of data at the boundary of the time interval in the continuation argument produces a logarithm in the final inequality. The first step which is already of logarithmic type for the elliptic case becomes of double logarithmic type in the parabolic case. The final global inequality is therefore of log-log-log-log type. In addition to that there is an additional difficulty due to the fact that the data at $t=t_0$ and $t=t_1$ are wanting. Precisely we obtain in each of the three steps described above a global estimate in $(t_0+\epsilon ,t_1-\epsilon)$, $\epsilon$ arbitrarily small. We then recover the missing data in $(t_0,t_0+\epsilon )$ and $(t_1+\epsilon ,t_1$ with the help of Hardy's inequality for vector-valued functions.

Even if the stability inequality we obtain seems to be very weak, the three-cylinder inequality appears to be the right tool for continuing a 
solution of a parabolic equation. For this reason, we are not convinced that 
Theorem \ref{theorem1} can be improved, in general, by using a global method. 

However we do not pretend that our stability inequality is the best possible. It is even an open problem to know whether the best possible stability inequality is of single logarithm type.

Although we used classical tools to establish our main result, 
there are no publications on the global stability in the sense of 
Theorem 1.1, and our proof is entirely self-contained. This is our  
contribution to the stability issue for parabolic Cauchy problems.

The most part in our analysis is  build on a  Carleman inequality 
(Theorem \ref{theorem1.1} below). We observe that Carleman inequalities are 
very useful in control theory and for establishing the unique 
continuation property for elliptic and parabolic partial differential 
equations. There is wide literature on this subject. We just quote here the following 
few references \cite{BY, Bo2, FLZ, Ch0,FG,FI,Hu,Is,LLe,LL,SS}.

The rest of this article is organized as follows. Section \ref{section2} is 
devoted to a three-cylinder interpolation inequality with respect to 
the $L_t^2(H_x^1)$-norm. This inequality will be very useful for continuing 
the data on an interior subdomain to the lateral boundary data, and to 
continue the data from one subdomain to another subdomain. 
This is what we show in Section \ref{section3} and, as byproduct, we prove a 
stability estimate corresponding to the unique continuation from an interior 
data. The proof of Theorem \ref{theorem1} is completed in Section 
\ref{section4} by beforehand establishing a result that quantifies 
the stability from the Cauchy data to an interior subdomain.

\section{Three-cylinder interpolation inequality}\label{section2}

We prove in this section 
\begin{theorem}\label{theorem2.1}
 There exist  $C>0$ and $0<\vartheta <1$, only depending on $\kappa$, $\Omega$ and $T_0$,  so that, for any  $0<\epsilon <(t_1-t_0)/2$, $u\in H^1((t_0,t_1),H^2(\Omega ))$ satisfying $Lu=0$ in $Q$,  $y\in \Omega $ and $0<r<r_y(\epsilon )=\min \left(\mbox{dist}(y,\Gamma )/3,\sqrt{\epsilon}\right)$, we have
\begin{align}
&r^3\|u\|_{L^2((t_0+\epsilon ,t_1-\epsilon ),H^1(B(y,2r)))} \label{2.1}
\\
&\hskip 2cm \le C \|u\|_{L^2((t_0,t_1),H^1(B(y,r)))}^\vartheta \|u\|_{L^2((t_0,t_1),H^1(B(y,3r)))}^{1-\vartheta}. \nonumber
\end{align}
\end{theorem}

The proof of Theorem \ref{theorem2.1} is based on a Carleman inequality for a family of parabolic operators. To this end, let $\mathcal{Z}$ be an arbitrary set and consider the family of operators 
\[
L_z=\mbox{div}(A_z\nabla \, \cdot\, )-\partial _t,\quad z\in \mathcal{Z},
\]
where, for each $z\in \mathcal{Z}$,  $A_z=(a_z^{ij})$ is a symmetric matrix with  $W^{1,\infty}(\Omega )$ entries and there exists $0<\kappa \le 1$  so that 
\begin{equation}\label{1.1}
A_z(x)\xi \cdot \xi \geq \kappa |\xi |^2,\quad  x\in \Omega , \; \xi \in \mathbb{R}^n\; \mbox{and}\; z\in \mathcal{Z},
\end{equation}
and 
\begin{equation}\label{1.2}
\|a_z^{ij}\|_{W^{1,\infty}(\Omega )}\le \kappa^{-1},\quad 1\leq i,j\leq n,\; z\in \mathcal{Z}.
\end{equation}

Pick $\psi \in C^2(\overline{\Omega})$ without critical points in $\overline{\Omega}$ and set $\Sigma =\Gamma \times (t_0,t_1)$. Let 
\[
g(t)=\frac{1}{(t-t_0)(t_1-t)}
\] 
and 
\begin{align*}
&\varphi (x,t) =g(t)\left( e^{4\lambda \|\psi\|_\infty}-e^{\lambda (2 \|\psi\|_\infty + \psi (x))}\right),
\\
& \chi (x,t)=g(t)e^{\lambda (2 \|\psi\|_\infty + \psi (x))}.
\end{align*}

\begin{theorem}\label{theorem1.1} (Carleman inequality)
There exist three positive constants $C$, $\lambda _0$ and $\tau _0$,
only depending only on  $\psi$, $\Omega$, $\kappa$ and $T_0$, so that
\begin{align}
&C\int_Q \left(\lambda ^4\tau ^3\chi ^3u^2+\lambda ^2\tau \chi |\nabla u|^2 \right)e^{-2\tau \varphi} dxdt \label{1.3}
\\
&\hskip 1cm \le \int_Q (L_zu)^2e^{-2\tau \varphi}dxdt \nonumber
\\
&\hskip 2cm+\int_\Sigma \left( \lambda^3\tau ^3\chi ^3u^2+\lambda \tau \chi |\nabla u|^2+(\lambda \tau \chi )^{-1}(\partial _tu)^2)\right)e^{-2\tau \varphi} d\sigma dt, \nonumber
\end{align}
for all $u\in H^1((t_0,t_1),H^2(\Omega ))$, $z\in \mathcal{Z}$, $\lambda \geq \lambda _0$ and $\tau \geq \tau _0$.\nonumber
\end{theorem}

\begin{proof}
Since the dependance of the constants will be uniform with respect to $z\in \mathcal{Z}$, we drop for simplicity the subscript $z$ in $L_z$ and its coefficients. On the other hand, as $C^\infty(\overline{Q})$ is dense in $H^1((t_0,t_1),H^2(\Omega ))$, it is enough to prove \eqref{1.3} when $u\in C^\infty(\overline{Q})$.

Let $\Phi =e^{\tau \varphi}$, $u\in C^\infty(\overline{Q})$  and set $w=\Phi^{-1}u$ that we extend by continuity at $t=0$ and $t=T$ by setting $w(\cdot ,0)=w(\cdot ,T)=0$. Then straightforward computations give
\[
Pw=[\Phi ^{-1}L \Phi ]w=P_1w+P_2w+cw,
\]
where
\begin{align*}
&P_1w=aw+\mbox{div}\, ( A\nabla w) -\tau \partial _t \varphi w,
\\
&P_2w= B\cdot \nabla w+bw -\partial _tw,
\end{align*}
with
\begin{align*}
&a=a(x,t,\lambda ,\tau )= \lambda ^2\tau ^2  \chi^2|\nabla \psi |_A^2,
\\
&B=B(x,t,\lambda ,\tau )=-2\lambda \tau  \chi A\nabla \psi ,
\\
&b=b(x,t,\lambda ,\tau )=-2\lambda ^2\tau \chi |\nabla \psi |_A^2,
\\
&c=c(x,t,\lambda ,\tau )=-\lambda \tau \chi\mbox{div}\, ( A\nabla \psi  )+\lambda ^2\tau  \chi |\nabla \psi |_A^2 .
\end{align*}
Here
\[
|\nabla \psi |_A=\sqrt{A\nabla \psi \cdot \nabla \psi}=\left|A^{1/2}\nabla u\right|.
\]
We obtain by  making integrations by parts
\begin{align}
\int_Q aw(B\cdot \nabla w)dxdt & =\frac{1}{2}\int_Q a(B\cdot \nabla w^2)dxdt \label{1.4}
\\
&=-\frac{1}{2}\int_Q \mbox{div}(aB) w^2dxdt+\frac{1}{2}\int_\Sigma a(B\cdot \nu) w^2d\sigma dt\nonumber
\end{align}
and
\begin{align}
&\int_Q \mbox{div}\, ( A\nabla w)B\cdot \nabla wdxdt \label{1.5}
\\
&\hskip 1cm =-\int_Q A\nabla w\cdot \nabla (B\cdot \nabla w)dxdt+\int_\Sigma (B\cdot \nabla w)(A\nabla w\cdot\nu )d\sigma dt\nonumber
\\
&\hskip 1 cm=-\int_Q (B')^t\nabla w\cdot A\nabla wdxdt\nonumber
\\
&\hskip 3cm -\int_Q\nabla ^2wB\cdot A\nabla wdxdt+\int_\Sigma (B\cdot \nabla w)( A\nabla w\cdot\nu ) d\sigma dt. \nonumber
\end{align}
Here $B'=(\partial _jB_i)$ is the Jacobian matrix of $B$ and $\nabla ^2w=(\partial ^2_{ij}w)$ is the Hessian matrix of $w$.

But
\begin{align*}
\int_Q B_j\partial ^2_{ij}w a^{ik}\partial _kwdxdt&=-\int_Q B_j a^{ik}\partial ^2_{jk}w \partial _i wdxdt
\\ 
&\quad -\int_Q\partial _j\left[B_ja^{ik}\right]\partial _kw\partial _i wdxdt+\int_\Sigma B_j\nu _j a^{ik}\partial _kw\partial _iwd\sigma dt.
\end{align*}
Therefore
\begin{align}
\int_Q\nabla ^2wB\cdot A\nabla wdxdt=-\frac{1}{2}\int_Q &\left(\left[\mbox{div}(B)A+\tilde{A}\right]\nabla w\right)\cdot \nabla wdxdt \label{1.6}
\\
&+\frac{1}{2}\int_\Sigma |\nabla w|_A^2(B\cdot \nu )  d\sigma dt ,\nonumber
\end{align}
where $\tilde{A}=(\tilde{a}^{ij})$ with $\tilde{a}^{ij}=B\cdot \nabla a^{ij}.$

It follows from \eqref{1.5} and \eqref{1.6} 
\begin{align}
\int_Q \mbox{div}\, ( A\nabla w)B&\cdot \nabla wdxdt=\frac{1}{2}\int_Q  \left(-2A(B')^t+\mbox{div}(B)A+\tilde{A} \right)\nabla w\cdot\nabla wdxdt \label{1.7}
\\
&+\int_\Sigma \left(B\cdot \nabla w\right) \left(A\nabla w\cdot\nu\right) d\sigma dt- \frac{1}{2}\int_\Sigma |\nabla w|_A^2(B\cdot \nu ) d\sigma dt. \nonumber
\end{align}

A new integration by parts yields
\begin{align*}
\int_Q \mbox{div}\, ( A\nabla w) bwdxdt=-\int_Q b|\nabla w|_A^2dxdt&-\int_Q w\nabla b\cdot A\nabla wdxdt
\\
&\qquad +\int_\Sigma bwA\nabla w\cdot \nu d\sigma dt.
\end{align*}
This and 
\[
-\int_Q w\nabla b\cdot A\nabla wdxdt\geq -\int_Q (\lambda ^2\chi )^{-1}|\nabla b|_A^2w^2dxdt-\int_Q \lambda ^2\chi |\nabla w|_A^2dx dt
\]
(where we used the inequality $|AX\cdot Y| \le |X|_A^2+|Y|_A^2$, for all $X,Y\in \mathbb{R}^n$) imply
\begin{align}
\int_Q \mbox{div}\, ( A\nabla w) bwdxdt &\ge -\int_Q (b+\lambda ^2\chi )|\nabla w|_A^2dxdt \label{1.8}
\\
&-\int_Q (\lambda ^2\chi )^{-1}|\nabla b|_A^2w^2dxdt+\int_\Gamma bw(A\nabla w\cdot \nu) d\sigma dt.\nonumber
\end{align}

One more time, integrations by parts entail
\begin{equation}\label{1.8.1}
\int_Qa w\partial _twdxdt=\frac{1}{2}\int_Qa \partial _tw^2dxdt=-\frac{1}{2}\int_Q\partial _t a w^2dxdt,
\end{equation}
\begin{equation}\label{1.8.2}
\int_Q\partial _t \varphi w\partial _twdxdt=\frac{1}{2}\int_Q\partial _t \varphi \partial _tw^2dxdt=-\frac{1}{2}\int_Q\partial _t ^2\varphi w^2dxdt,
\end{equation}
\begin{align}
\int_Q \partial_t \varphi wB\cdot \nabla wdxdt &= \frac{1}{2}\int_Q\partial_t \varphi B\cdot \nabla w^2dxdt  \label{1.8.3}
\\
&=-\frac{1}{2}\int_Q\mbox{div}(\partial_t \varphi B) w^2dxdt +\frac{1}{2} \int_\Sigma \partial _t\varphi (B\cdot \nu )w^2d\sigma dt. \nonumber
\end{align}
Also,
\begin{equation}\label{1.8.4}
\int_Q\mbox{div}(A\nabla w)\partial _twdxdt=-\int_Q A\nabla w\cdot \nabla \partial_twdxdt+\int_\Sigma (A\nabla w\cdot \nu) \partial _tw d\sigma dt.
\end{equation}
But an integration by parts with respect to $t$ gives
\[
\int_Q A\nabla w\cdot \nabla \partial_twdxdt=-\int_Q A\nabla \partial _tw\cdot \nabla wdxdt= -\int_Q \nabla \partial _tw\cdot A\nabla wdxdt,
\]
where we used $w(\cdot ,0)=w(\cdot ,T)=0$.

Whence
\[
\int_Q A\nabla w\cdot \nabla \partial_twdxdt=0.
\]
This identity in \eqref{1.8.4} entails
\begin{equation}\label{1.8.5}
\int_Q\mbox{div}(A\nabla w)\partial _tw=\int_\Sigma (A\nabla w\cdot \nu) \partial _tw d\sigma dt.
\end{equation}

Now a combination of \eqref{1.4}, \eqref{1.7} to \eqref{1.8.3} and  \eqref{1.8.5} gives
\begin{align}
\int_Q P_1wP_2wdxdt &-\int_Q c^2w^2dxdt \label{1.9}
\\
&\ge \int_Q fw^2dxdt+\int_Q F\nabla w\cdot \nabla w dxdt+\int_\Sigma g(w)d\sigma dt,\nonumber
\end{align}
where
\begin{align*}
&f=-\frac{1}{2}\mbox{div}(aB)+ab-(\lambda ^2\chi )^{-1}|\nabla b|_A^2-c^2 +\frac{1}{2}\partial _ta -\frac{\tau}{2}\partial_t^2\varphi +\frac{\tau}{2}\mbox{div}(\partial _t\varphi B)-\tau b\partial_t\varphi ,
\\
&F=-A(B')^t+\frac{1}{2}\Big(\mbox{div}(B)A +\tilde{A}\Big) -(b+\lambda ^2\chi )A,
\\
&g(w)=\frac{1}{2}aw^2(B\cdot \nu)-\frac{1}{2}|\nabla w|_A^2(B\cdot \nu)+(B\cdot \nabla w)(A\nabla w \cdot \nu) 
\\
&\hskip 4cm +bw(A\nabla w \cdot \nu) -\frac{\tau}{2}\partial _t\varphi (B\cdot \nu) w^2 -(A\nabla w\cdot \nu )\partial _tw.
\end{align*}

We obtain, by using the elementary  inequality $(\alpha -\beta )^2\geq \alpha^2/2-\beta^2$,  
\begin{align*}
\|Pw\|_2^2 &\geq (\|P_1w+P_2w\|_2-\|cw\|_2)^2\\ &\geq \frac{1}{2}\|P_1w+P_2w\|_2^2-\|cw\|_2^2\\ &\geq \int_\Omega P_1wP_2w dx-\int_\Omega c^2w^2dx.
\end{align*}
Whence \eqref{1.9} implies
\begin{equation}\label{1.10}
\|Pw\|_2^2\geq \int_Q fw^2dxdt+\int_Q F\nabla w\cdot \nabla w dxdt+\int_\Sigma g(w)d\sigma dt.
\end{equation}

In light of the following inequalities, where $C$ is a constant depending only on $T_0$ and $\psi$,
\begin{align*}
&|\partial _t  \varphi |\le C\chi ^2,
\\
& |\partial _t^2\varphi |,\; |\nabla \partial _t\varphi |\le C\chi ^3
\\
&|(A\nabla w\cdot \nu) \partial _tw|\le \lambda \tau \chi |A\nabla w\cdot \nu|^2+(\lambda \tau \chi )^{-1}(\partial _tw)^2,
\end{align*}
straightforward computations show that there exist four positive constants $C_0$, $C_1$, $\lambda _0$ and $\tau_0$, only depending only on $\psi$, $\Omega$, $T_0$ and $\kappa$, such that, for all $\lambda \geq \lambda _0$ and $\tau \geq \tau_0$, so that
\begin{align*}
&f\geq C_0 \lambda ^4\tau ^3\chi ^3,
\\
&F\xi \cdot \xi \geq C_0\lambda ^2\tau \chi |\xi |^2,\;\; \textrm{for any}\; \xi \in \mathbb{R}^n,
\\
&|g(w)|\leq C_1\left( \lambda ^3\tau ^3\chi ^3w^2+\lambda \tau \chi |\nabla w|^2 +(\lambda \tau \chi )^{-1}(\partial _tw)^2\right).
\end{align*}
Hence
\begin{align}
C\int_Q  (\lambda ^4\tau ^3\chi ^3w^2&+\lambda ^2\tau \chi |\nabla w|^2 ) dxdt \leq \int_Q (Pw)^2dxdt \label{1.11}
\\ 
&+\int_\Sigma ( \lambda^3\tau ^3\chi ^3w^2+\lambda \tau \chi |\nabla w|^2+(\lambda \tau \chi )^{-1}(\partial _tw)^2) d\sigma dt.\nonumber
\end{align}

As $\nabla w=\Phi^{-1}\left( \nabla u+\lambda \tau \chi u\nabla \psi \right)$, we obtain
\[
|\nabla w|^2= \Phi ^{-2}\left( |\nabla u|^2+\lambda ^2\tau ^2\chi ^2|\nabla \psi |^2u^2+2\lambda \tau \chi u \nabla u\cdot \nabla \psi \right).
\]
Therefore we find, by using an elementary inequality,
\[
|\nabla w|^2\ge \Phi ^{-2}\left( |\nabla u|^2+\lambda ^2\tau ^2|\nabla \psi |^2u^2-4\lambda ^2 \tau^2  u^2 | \nabla \psi |^2-\frac{1}{2}|\nabla u|^2\right)
\]
and then

\[
|\nabla w|^2\ge \Phi ^{-2}\left(\frac{1}{2} |\nabla u|^2-3\lambda ^2 \tau^2\chi ^2  u^2 \| \nabla \psi \|_\infty^2\right).
\]

Consequently, modifying $\lambda_0$ if necessary, we get
\begin{equation}\label{1.13}
\lambda ^2\tau \chi |\nabla w|^2 +\lambda ^4 \tau^3\chi ^3  w^2\ge C \Phi ^{-2}\left( \lambda ^2\tau \chi|\nabla u|^2+\lambda ^4 \tau^3\chi ^3 \tau^2  u^2\right).
\end{equation}
On the other hand, it is not hard to establish the inequality
\begin{equation}\label{1.14}
(\partial _tw)^2\le \Phi ^{-2}\left((\partial_tu)^2+C\tau ^2\chi ^2u^2\right).
\end{equation}
The expected inequality follows then by combining \eqref{1.11}, \eqref{1.13} and \eqref{1.14}.
\end{proof}

From the preceding proof it is obvious that  Theorem \ref{theorem1.1} holds whenever $L_z$ is replaced by $L$. That is we have
\begin{theorem}\label{theorem1.1+} (Carleman inequality)
There exist three positive constants $C$, $\lambda _0$ and $\tau _0$,
only depending only on  $\psi$, $\Omega$, $\kappa$ and $T_0$, so that
\begin{align}
&C\int_Q \left(\lambda ^4\tau ^3\chi ^3u^2+\lambda ^2\tau \chi |\nabla u|^2 \right)e^{-2\tau \varphi} dxdt \label{1.3+}
\\
&\hskip 1cm \le \int_Q (Lu)^2e^{-2\tau \varphi}dxdt \nonumber
\\
&\hskip 2cm+\int_\Sigma \left( \lambda^3\tau ^3\chi ^3u^2+\lambda \tau \chi |\nabla u|^2+(\lambda \tau \chi )^{-1}(\partial _tu)^2)\right)e^{-2\tau \varphi} d\sigma dt, \nonumber
\end{align}
for all $u\in H^1((t_0,t_1),H^2(\Omega ))$, $\lambda \geq \lambda _0$ and $\tau \geq \tau _0$.\nonumber
\end{theorem}

\begin{proof}[Proof of Theorem \ref{theorem2.1}]
Let $u\in H^1((t_0,t_1),H^2(\Omega ))$ satisfying $Lu=0$ and set 
\[
Q(\mu )=B(0,\mu )\times (-1,1),\quad \mu >0.
\]
Fix $(y,s)\in \Omega \times (t_0,t_1)$ and  
\[
0<r< r_{(y,s)}=\min \left(\mbox{dist}(y,\Gamma )/3, \sqrt{s-t_0},\sqrt{t_1-s}\right)\le r_0=r_0(\mbox{diam}(\Omega ),T_0 ).
\] 

Let
\[
w(x,t)=u(rx+y,r^2t +s),\quad (x,t)\in Q(3),
\]
Then 
\[
L_rw=\textrm{div}(A_r\nabla w )-\partial _tw=0\quad  \mbox{in}\; Q(3),
\]
where $A_r(x)=(a^{ij}(rx+y))$.
 
Clearly, the family $(A_r)$  satisfies \eqref{1.1} and \eqref{1.2} uniformly with respect to $r\in (0, r_{(y,s)})$.

Let $\phi \in C_0^\infty (U)$ satisfying $0\leq \phi \leq 1$ and  $\phi =1$ in $\mathcal{K}$, with
\[
U=\left\{x\in \mathbb{R}^n;\; 1/2<|x|<3\right\}\quad \mbox{and}\quad \mathcal{K}=\left\{x\in \mathbb{R}^n;\; 1\leq |x|\leq 5/2\right\}.
\]
Theorem \ref{theorem1.1} applied to $\phi w$ when $\Omega$ is substituted by $U$ and $g(t)=1/(1-t^2)$ gives, for $\lambda \geq \lambda _0$ and $\tau \geq \tau _0$,
\begin{align}
&C\int_{Q(2)\setminus Q(1)} \left (\lambda ^4\tau ^3\varphi ^3w^2+\lambda ^2\tau \varphi |\nabla w|^2 \right)e^{-2\tau \varphi} dxdt \label{1.17}
\\
&\hskip 6cm \leq \int_{Q(3)} (L_r(\phi w))^2e^{-2\tau \varphi}dxdt, \nonumber
\end{align}
 the constant $C$ only depends on  $\kappa$.
 
But
\[
\mbox{supp}(L_r(\phi w))\subset \left[\left\{1/2 \leq |x|\leq 1\right\}\cup \left\{5/2\leq |x|\leq 3\right\}\right]\times (-1,1)
\]
and
\[
(L_r(\phi w))^2 \leq \Lambda (w^2+|\nabla w|^2 ),
\]
where $\Lambda =\Lambda (r_0)$ is independent on $r$. Therefore, fixing $\lambda$ and changing $\tau _0$ if necessary, \eqref{1.17} implies, for $\tau \geq \tau _0$,
\begin{align}
C\int_{Q(2)} &\left (w^2+|\nabla w|^2 \right)e^{-2\tau \varphi} dxdt\leq \int_{Q(1)} \left (w^2+|\nabla w|^2 \right)e^{-2\tau \varphi} dxdt\label{1.18}
\\
&\hskip 2cm +\int_{Q(3)\setminus Q(5/2)} \left (w^2+|\nabla w|^2 \right)e^{-2\tau \varphi} dxdt.\nonumber
\end{align}

Let  $0<\rho <1$ to be specified later and choose $\psi (x)=-|x|^2$ in \eqref{1.18} (which is without critical points in $U$). In that case
\[
\varphi (x,t) =g(t)\left( e^{36\lambda }-e^{\lambda (18 - |x|^2)}\right).
\]
We have
\begin{align*}
&\varphi (x,t)\le g(-1+\rho )\left( e^{36\lambda }-e^{14\lambda}\right)\le \frac{1}{\rho }\left( e^{36\lambda }-e^{14\lambda}\right)=\frac{\alpha}{\rho } ,
\\
&\hskip 6cm  (x,t)\in B(2)\times (-1+\rho, 1-\rho ),
\\
&\varphi (x,t)\ge g(0)\left(e^{36\lambda}-e^{18\lambda}\right)=\left(e^{36\lambda}-e^{18\lambda}\right)=\beta,\quad (x,t)\in Q(1),
\\
&\varphi (x,t)\ge g(0)\left(e^{36\lambda}-e^{\frac{47}{4}\lambda}\right)=\left(e^{36\lambda}-e^{\frac{47}{4}\lambda}\right)=\gamma ,\quad (x,t)\in Q(3)\setminus Q\left(5/2\right).
\end{align*}

As $\frac{\beta}{\alpha}<1<\frac{\gamma}{\alpha}$, we can fix $\theta \in (0,1)$ so that 
\[
\frac{1}{\rho}:=\theta \frac{\beta}{\alpha} +(1-\theta )\frac{\gamma}{\alpha} >1 .
\]
Set $a=2(1-\theta )(\gamma -\beta )$ and $b=2\theta (\gamma -\beta )$ and $\tilde{Q}(2)=B(0,2)\times (-1+\rho ,1-\rho  )$. Then \eqref{1.18} yields
\begin{align*}
C\int_{\tilde{Q}(2)} &\left (w^2+|\nabla w|^2 \right) dxdt
\\
&\le e^{a\tau } \int_{Q(1)} \left (w^2+|\nabla w|^2 \right) dxdt +e^{-b\tau }\int_{Q(3)} \left (w^2+|\nabla w|^2 \right) dxdt.
\end{align*}
Similarly to the elliptic case \cite[Theorem 2.17 and its proof, pages 19 to 21]{Ch} (see also the proof of Proposition \ref{propositionS4}), we obtain from this inequality the following one
\[
C\|w\|_{L^2((-1+\rho ,1-\rho ),H^1(B(2)))}\le \|w\|_{L^2((-1,1 ),H^1(B(1)))}^\vartheta \|w\|_{L^2((-1,1 ),H^1(B(3)))}^{1-\vartheta},
\]
with $\vartheta=\frac{a}{a+b}$.

We get by making a change of variable, where $\tau =1-\rho$,
\begin{align}
&r\|u\|_{L^2((s-\tau r^2 ,s+\tau r^2 ),H^1(B(y,2r)))} \label{3s}
\\
&\qquad \le C \|u\|_{L^2((s-r^2,s+r^2 ),H^1(B(y,r)))}^\vartheta \|u\|_{L^2((s-r^2,s+r^2),H^1(B(y,3r)))}^{1-\vartheta}.\nonumber
\end{align}
Here and until the end of this proof, the generic constant $C$ only depends on $\Omega$, $\kappa$ and $T_0$.

Fix $0<\epsilon < (t_1-t_0)/2$. Let $s_0=t_0+\epsilon$ and $s_k=s_{k-1}+2\tau r^2$, $k\ge 1$, in such a way that
\[
(s_{k-1},s_k)=((s_{k-1}+\tau r^2)-\tau r^2, (s_{k-1}+\tau r^2)+\tau r^2).
\]

We consider $q$ the smallest integer so that $(t_1-\epsilon )-s_{q-1}\le 2\tau r^2$ or equivalently $(t_1-\epsilon ) - s_{q-2}>2\tau r^2$. Whence
\begin{equation}\label{4s}
q<\frac{t_1-t_0-2\epsilon}{2\tau r^2}+3<\frac{\delta}{2\tau r^2}+\frac{3\mbox{diam}(\Omega )^2}{r^2}= \left(\frac{T_0}{2\tau }+3\mbox{diam}(\Omega )^2\right)\frac{1}{r^2}.
\end{equation}

Let $r<r_y(\epsilon )=\min \left(\mbox{dist}(y,\Gamma )/3,\sqrt{\epsilon}\right)$. It follows from \eqref{3s} that
\begin{align*}
&r\|u\|_{L^2((s_{k-1} ,s_k),H^1(B(y,2r)))} 
\\
&\qquad \le C \|u\|_{L^2((t_0,t_1 ),H^1(B(y,r)))}^\vartheta \|u\|_{L^2((t_0,t_1),H^1(B(y,3r)))}^{1-\vartheta},
\end{align*}
with $s_q=t_1-\epsilon$.

Thus
\begin{align*}
&r\sum_{k=1}^{q}\|u\|_{L^2((s_{k-1} ,s_k),H^1(B(y,2r)))} 
\\
&\qquad \le Cq \|u\|_{L^2((t_0,t_1 ),H^1(B(y,r)))}^\vartheta \|u\|_{L^2((t_0,t_1),H^1(B(y,3r)))}^{1-\vartheta}.
\end{align*}

In consequence  
\begin{align}
&r\|u\|_{L^2((t_0+\epsilon ,t_1-\epsilon ),H^1(B(y,2r)))}\label{5s}
\\
&\hskip 1cm \le Cq \|u\|_{L^2((t_0,t_1),H^1(B(y,r)))}^\vartheta \|u\|_{L^2((t_0,t_1),H^1(B(y,3r)))}^{1-\vartheta}.\nonumber
\end{align}

Estimate \eqref{4s} in \eqref{5s} yields
\begin{align*}
&Cr^3\|u\|_{L^2((t_0+\epsilon ,t_1-\epsilon ),H^1(B(y,2r)))}
\\
&\hskip 1cm \le  \|u\|_{L^2((t_0,t_1),H^1(B(y,r)))}^\vartheta \|u\|_{L^2((t_0,t_1),H^1(B(y,3r)))}^{1-\vartheta},\quad r<r_y(\epsilon ). 
\end{align*}
The proof is then complete.
\end{proof}

\section{Quantifying the uniqueness of continuation from an interior data}\label{section3}

We start with a Hardy inequality for vector valued functions.

\begin{lemma}\label{lemmaH}
Let $X$ be a Banach space with norm  $\|\cdot \, \|$ and $s\in (0,1/2)$. There exists a constant $c>0$ so that, for any $u\in H^s((t_0,t_1),X)$, we have
\[
\left\| \frac{u}{\delta ^s}\right\|_{L^2((t_0,t_1),X)}\le c\|u\|_{H^s((t_0,t_1),X)}.
\]
Here $\delta =\delta (t)=\min\{ |t-t_0|,|t-t_1|\}$.
\end{lemma}

\begin{proof}
 Let $u\in H^s((t_0,t_1),X)$. From the usual Hardy's inequality in dimension one (see for instance \cite{Gr}) we have
\begin{equation}\label{H1}
\int_{t_0}^{t_1}\frac{\|u(t)\|^2}{\delta ^{2s}(t)}dt\le c\| \| u(\cdot )\| \|^2_{H^s((t_0,t_1))}.
\end{equation}
But
\begin{align*}
\| \| u(\cdot )\| \|^2&_{H^s((t_0,t_1))}=\|u\|^2_{L^2((t_0,t_1),X)}+\int_{t_0}^{t_1}\int_{t_0}^{t_1}\frac{|\|u(\tau)\| -\|u(t)\||^2}{|\tau -t|^{1+2s}}dtd\tau
\\
&\le \|u\|^2_{L^2((t_0,t_1),X)}+\int_{t_0}^{t_1}\int_{t_0}^{t_1}\frac{\|u(\tau)- u(t)\|^2}{|\tau -t|^{1+2s}}dtd\tau =\|u\|^2_{H^s((t_0,t_1),X)}.
\end{align*}
Whence the result follows.
\end{proof}

In the rest of this paper we shall often apply Hardy's inequality in Lemma \ref{lemmaH} to functions from $H^k((t_0,t_1),H)$, where $k\ge 1$ is an integer and $H$ is a Hilbert space. This is made possible by \cite[Remark 9.5, page 46]{LM} saying that $H^s((t_0,t_1),H)$, $0<s<1$, can be seen as an interpolated space between $L^2((t_0,t_1),H)$ and $H^1((t_0,t_1),H)$. Precisely, we have
\[
H^s((t_0,t_1),H)=[L^2((t_0,t_1),H),H^1((t_0,t_1),H)]_{1-s},\quad 0<s<1.
\]
We readily obtain from Lemma \ref{lemmaH} the following corollary.

\begin{corollary}\label{corollaryH}
Let $H$ be a Hilbert space and $s\in (0,1/2)$. There exists a constant $c>0$ so that, for any $u\in H^1((t_0,t_1),H)$, we have
\[
\left\| \frac{u}{\delta ^s}\right\|_{L^2((t_0,t_1),H)}\le c\|u\|_{H^1((t_0,t_1),H)},
\]
where $\delta$ is as in Lemma \ref{lemmaH}.
\end{corollary}

Next, we prove

\begin{proposition}\label{proposition3.1}
Let $s\in (0,1/2)$. There exist $\omega \Subset \Omega$, only depending on $\Omega$, and  three constants $c>0$, $C>0$ and $\sigma _0>0$, only depending on $\Omega$, $\kappa$, $T_0$, $s$ and $\alpha$, so that, for any $u\in \mathscr{X}(Q)$ satisfying $Lu=0$ in $Q$ and $0<\sigma <\sigma _0$, we have
\begin{align}
\|u\|_{L^2((t_0,t_1),L^2(\Gamma ))}&+\|\nabla u\|_{L^2((t_0,t_1),L^2(\Gamma ))} \label{3.1}
\\
&\le C\left(\sigma ^{\min\{\alpha ,s\}/2} \|u\|_{\mathscr{X}(Q)} +e^{e^{c/\sqrt{\sigma}}}\|u\|_{L^2((t_0,t_1),H^1(\omega ))}\right).\nonumber
\end{align}
\end{proposition}

\begin{proof}
Since $\Omega$ is Lipschitz, it has the uniform interior cone property (see for instance \cite{HP}). That is there exist $R>0$ and $\theta \in \left]0,\frac{\pi}{2}\right[$ so that, for any $\tilde{x}\in \Gamma$, we may find $\xi =\xi (\tilde{x})\in \mathbb{S}^{n-1}$ for which
\[
\mathcal{C}(\tilde{x})=\{x\in \mathbb{R}^n;\; 0<|x-\tilde{x}|<R,\; (x-\tilde{x})\cdot \xi >|x-\tilde{x}|\cos \theta \}\subset \Omega .
\]

Fix $\tilde{x}\in \Gamma$ and let $\xi =\xi (\tilde{x})$ be as in the definition above. 

Let $0<\epsilon \le \epsilon_0<\min \left\{ ((3R/(2\sin \theta ))^2,(t_1-t_0)/4\right\}$, and set  $y_0=y_0(\tilde{x})=\tilde{x}+(R/2)\xi$ and $\rho=\sqrt{\epsilon}\sin \theta /3$. Let $N=[R/(2\rho)]$, the integer part of $R/(2\rho)$, if $R/(2\rho)\not\in \mathbb{N}$ and $N=R/(2\rho)-1$ if $R/(2\rho)\in \mathbb{N}$. Define then
\[
x_0=\tilde{x}+(R/2-N\rho)\xi .
\]
Furthermore, consider the sequence
\[
y_j=\tilde{x}+(R/2-j\rho)\xi ,\quad 0\le j\le N.
\]
By construction, $B(y_j,3\rho)\subset \mathcal{C}(\tilde{x})$, $0\le j\le N$ and, as $|y_{j+1}-y_j|=\rho$, we have
\[
B(y_{j+1},\rho)\subset B(y_j,2\rho),\quad 0\le j\le N-1.
\]

Let $u\in \mathscr{X}(Q)$. We use in the sequel the temporary notation
\[
M=M(u)=\|u\|_{\mathscr{X}(Q)}.
\] 

Set $I_j=(t_0^j ,t_1^j)$, where $t_i^j=t_i+(-1)^ij\epsilon$, with $i=0,1$ and $0\le j\le N$. Note that \[ N\epsilon \le (R/(2\rho))\epsilon=(3R/(2\sin \theta ))\sqrt{\epsilon}.\] Then $I_N\ne \emptyset$ if $(3R/\sin \theta ) \sqrt{\epsilon}<t_1-t_0$. This condition always holds provided that we substitute $\epsilon_0$ by $\min (\epsilon_0, (t_1-t_0)^2\sin^2\theta /(9R^2))$.

In the rest of this proof $C$ is a generic constant only depending on $\Omega$, $\kappa$, $\alpha$, $s$ and $T_0$.

Using that $I_{j+1}=(t_0^j+\epsilon ,t_1^j-\epsilon )$ and noting that $\rho<\sqrt{\epsilon}$, we get from \eqref{2.1}
\begin{align*}
\rho^3&\|u\|_{L^2(I_{j+1},H^1(B(y_j,\rho )))}
\\
&\hskip 2cm \le CM^{1-\vartheta} \|u\|_{L^2(I_j,H^1(B(x_j,\rho )))}^\vartheta ,\quad 0\le j\le N-1.
\end{align*}
the constant $\vartheta$, only depending  on $\Omega$, $\kappa$ and $T_0$, satisfies $0<\vartheta <1$.

Whence
\[
\|u\|_{L^2(I_N,H^1(B(y_N,\rho )))}\le (C\rho^{-3})^{(1-\beta)/(1-\vartheta)}M^{1-\beta}\|u\|_{L^2(I_0,H^1(B(y_0,\rho )))}^{\beta},
\] 
with $\beta =\vartheta ^{N+1}$.

In this inequality, modifying $C$ if necessary, we may assume that $C\rho^{-3}\ge 1$. Thus
\[
\|u\|_{L^2(I_N,H^1(B(y_N,\rho )))}\le C\rho^{-3/(1-\vartheta)}M^{1-\beta}\|u\|_{L^2(I_0,H^1(B(y_0,\rho )))}^{\beta},
\]

Let $J=I_N$. Since $B(y_0,\rho)\subset B(y_0,R\sin \theta /6)\subset \mathcal{C}(\tilde{x})$ and $y_N=x_0$, the last inequality entails
\begin{equation}\label{+1}
\|u\|_{L^2(J,H^1(B(x_0,\rho )))}\le C\rho^{-3/(1-\vartheta)}M^{1-\beta}\|u\|_{L^2(I_0,H^1(B(y_0,R\sin \theta /6)))}^{\beta}.
\end{equation}

Define
\[
\omega =\bigcup_{\tilde{y}\in \Gamma}B(y_0(\tilde{y}),R\sin \theta /6).
\]
It is worth mentioning that $\omega$ only depends on $\Omega$.

We get from \eqref{+1}
\begin{equation}\label{+2}
\|u\|_{L^2(J,H^1(B(x_0,\rho )))}\le C\rho^{-3/(1-\vartheta)}M^{1-\beta}\|u\|_{L^2(I_0,H^1(\omega)))}^{\beta}.
\end{equation}

Now, since $u$ is H\"older continuous, we have
\[
|u(\tilde{x},t)|\le [u]_\alpha |\tilde{x}-x|^\alpha +|u(x,t)|,\quad x\in B(x_0,\rho ),\; t\in J.
\]
Here and henceforth
\[
[w]_\alpha =\sup_{\underset{(x_1,t_1)\neq (x_2,t_2)}{(x_1,t_1), (x_2,t_2)\in \overline{Q}}}\; \frac{|w(x_1,t_1)-w(x_2,t_2)|}{|x_1-x_2|^\alpha +|t_1-t_2|^{\alpha /2}},\quad w\in C^{\alpha ,\alpha /2}(\overline{Q}).
\]

Whence
\begin{align}
|\mathbb{S}^{n-1}|\rho^n\int_{J}|u(\tilde{x},t)|^2dt&\le 2n[u]_\alpha^2\int_{B(x_0,\rho )\times J} |\tilde{x}-x|^{2\alpha}dxdt \label{24s}
\\
&\hskip 1.5cm + 2n\int_{B(x_0,\rho )\times J} |u(x,t)|^2dxdt, \nonumber
\end{align}
Similarly, where $1\le i\le n$,
\begin{align}
|\mathbb{S}^{n-1}|\rho^n\int_{J}|\partial _iu(\tilde{x},t)|^2dt&\le 2n[\partial _iu]_\alpha^2\int_{B(x_0,\rho )\times J} |\tilde{x}-x|^{2\alpha}dxdt\label{25s}
\\
&\hskip 1.5cm + 2n\int_{B(x_0,\rho )\times J} |\partial _iu(x,t)|^2dxdt.\nonumber
\end{align}

We have
\begin{equation}\label{26s}
|\tilde{x}-x|\leq |\tilde{x}-x_0|+|x_0-x|\le  \rho+ (R/2-N\rho)\le 2\rho.
\end{equation}

Therefore, we have as a consequence of a combination of \eqref{24s}, \eqref{25s} and \eqref{26s}
\[
\int_{J}|u(\tilde{x},t)|^2dt+\int_{J}|\nabla u(\tilde{x},t)|^2dt\le C\left( M^2\rho^{2\alpha}+\rho^{-n}\|u\|_{L^2(J,H^1(B(x_0,\rho )))}\right)
\]
which, in light of \eqref{+2}, yields
\begin{align*}
\int_{J}|u(\tilde{x},t)|^2dt&+\int_{J}|\nabla u(\tilde{x},t)|^2dt\le 
\\
&C\left( M^2\rho^{2\alpha}+\rho^{-n-6/(1-\vartheta)}M^{2(1-\beta)}\|u\|_{L^2(I_0,H^1(\omega))}^{2\beta}\right).
\end{align*}
Integrating over $\Gamma$ both sides of this inequality with respect to $\tilde{x}$, we find
\begin{align}
\|u\|&_{L^2(\Gamma \times J)}+\|\nabla u\|_{L^2(\Gamma \times J)}\le \label{+4}
\\
&C\left( M\rho^{\alpha}+\rho_0^{-n/2-3/(1-\vartheta)}M^{1-\beta}\|u\|_{L^2(I_0,H^1(\omega))}^{\beta}\right).\nonumber
\end{align}

Bearing in mind that $J=(t_0+N\epsilon ,t_1+N\epsilon )$, we get by applying Corollary \ref{corollaryH}, for some fixed $s\in \left( 0,1/2\right)$,
\begin{align*}
&\|u\|_{L^2((t_0,t_0+N\epsilon),L^2(\Gamma ))} \le  c(N\epsilon)^s\|u\|_{H^1((t_0 ,t_1),L^2(\Gamma ))},
\\
&\|u\|_{L^2((t_1-N\epsilon ,t_1),L^2(\Gamma ))}\le  c(N\epsilon)^s\|u\|_{H^1((t_0 ,t_1),L^2(\Gamma ))},
\\
&\|\nabla u\|_{L^2((t_0,t_0+N\epsilon ),L^2(\Gamma ))}\le c(N\epsilon)^s \|\nabla  u\|_{H^1((t_0 ,t_1),L^2(\Gamma ))},
\\
&\|\nabla  u\|_{L^2((t_1-N\epsilon ,t_1),L^2(\Gamma ))}\le  c(N\epsilon)^s \|\nabla  u\|_{H^1((t_0 ,t_1),L^2(\Gamma ))}.
\end{align*}

Therefore, as the trace operator
\[
u\in H^1((t_0,t_1),H^2(\Omega ))\rightarrow (u,\nabla u)_{|\Sigma }\in H^1((t_0,t_1),L^2(\Gamma ))^{n+1}, 
\]
is bounded, we obtain
\begin{align*}
&\|u\|_{L^2((t_0,t_0+N\epsilon),L^2(\Gamma ))}\le c(N\epsilon)^s\|u\|_{H^1((t_0 ,t_1),H^2(\Omega ))},
\\
&\|u\|_{L^2((t_1-N\epsilon ,t_1),L^2(\Gamma ))}\le  c(N\epsilon)^s\|u\|_{H^1((t_0 ,t_1),H^2(\Omega ))},
\\
&\|\nabla u\|_{L^2((t_0,t_0+N\epsilon ),L^2(\Gamma ))}\le c(N\epsilon)^s \|u\|_{H^1((t_0 ,t_1),H^2(\Omega ))},
\\
&\|\nabla  u\|_{L^2((t_1-N\epsilon ,t_1),L^2(\Gamma ))}\le  c(N\epsilon)^s \|u\|_{H^1((t_0 ,t_1),H^2(\Omega ))}.
\end{align*}

These inequalities together  with \eqref{+4} give
\begin{align*}
&C\left(\|u\|_{L^2(\Sigma )}+\|\nabla u\|_{L^2(\Sigma )}\right)\le (\rho^{\alpha}+(N\epsilon)^s)M
\\
&\hskip 3cm +\rho^{-n/2-3/(1-\vartheta)}M^{1-\beta}\|u\|_{L^2(I_0,H^1(\omega))}^{\beta}.
\end{align*}
We obtain by applying Young's inequality to the last term 
\begin{align}
&C\left(\|u\|_{L^2(\Sigma )}+\|\nabla u\|_{L^2(\Sigma )}\right)\le (\rho^{\alpha}+(N\epsilon)^s+\epsilon^{\alpha/2})M \label{+5}
\\
&\hskip 1cm +\rho^{-(n/2+3/(1-\vartheta))/\beta}\epsilon^{-(1-\beta)\alpha/(2\beta)}\|u\|_{L^2(I_0,H^1(\omega))}.\nonumber
\end{align}
Next, we have
\begin{equation}\label{+6}
\rho^{\alpha}+(N\epsilon)^s+\epsilon^{\alpha/2}\le C\epsilon^{\min\{\alpha ,s\}/2}
\end{equation}
and, as $\beta=\theta ^{N+1}$, we have $\beta=O(e^{-c/\sqrt{\epsilon}})$, from which we deduce in a straightforward manner that
\begin{equation}\label{+7}
\rho^{-(n/2+3/(1-\vartheta))/\beta}\epsilon^{-(1-\beta)\alpha/(2\beta)}\le Ce^{e^{c\sqrt{\epsilon}}}.
\end{equation}
We end up by observing that \eqref{+6} and \eqref{+7} in \eqref{+5} give the expected inequality.
\end{proof}

The a priori estimate in the following lemma is well adapted to our purpose. It does not involve neither the initial time data nor the final time data.

\begin{lemma}\label{lemma3.1}
There exists a constant $C>0$, only depending on $\Omega$, $\kappa$ and $T_0$, so that, for any $u\in H^1((t_0,t_1),H^2(\Omega ))$ satisfying $Lu=0$ in $Q$, we have
\begin{equation}\label{3.2}
C\|u\|_{L^2((t_0 ,t_1),H^1(\Omega ))}\le \|u\|_{H^1((t_0,t_1),L^2(\Gamma ))}+\|\nabla u\|_{L^2((t_0,t_1),L^2(\Gamma ))}.
\end{equation}
\end{lemma}

\begin{proof} 
In this proof $C$ is a generic constant that can only depend on $\Omega$, $\kappa$ and $T_0$.

Let $u\in H^1((t_0,t_1),H^2(\Omega ))$ satisfying $Lu=0$ in $Q$ and set $v=e^{-t}u$. Then $v$ solves the following equation
\begin{equation}\label{32s}
\mbox{div}(A\nabla v)-v-\partial_tv=0\;\; \mbox{in}\; Q.
\end{equation}
Let $0<\epsilon <(t_1-t_0)/2$ and choose $\chi \in C_0^\infty ((t_0,t_1))$ satisfying $0\le \chi \le 1$, $\chi =1$ in $(t_0+\epsilon ,t_1-\epsilon )$ and, for some universal constant $c$, $|\chi '|\le c/\epsilon$.

We multiply \eqref{32s} by $\chi v$ and integrate over $Q$. We then get by making an integration by parts 
\[
-\int_Q \chi A\nabla v \cdot \nabla vdxdt-\int_Q\chi v^2dxdt+\int_\Sigma \chi v A\nabla v\cdot \nu d\sigma dt +\frac{1}{2}\int_Q v^2\chi ' dxdt =0,
\]
from which we deduce in a straightforward manner
\begin{align}
\int_Q \chi A\nabla u \cdot \nabla udxdt&+\int_Q\chi u^2dxdt \le \label{33s}
\\
&e^{2(t_1-t_0)}\left( C\int_\Sigma (u^2+|\nabla u|^2)d\sigma dx +\frac{1}{2}\int_Q u^2|\chi ' |dxdt\right).\nonumber
\end{align}
On the other hand, as $\mbox{supp}(\chi ')\subset (t_0,t_1)\setminus (t_0+\epsilon ,t_1-\epsilon )$, we have

\[
\mathscr{J}_\epsilon^2=\int_Q u^2|\chi ' |dxdt\le \frac{c}{\epsilon} \int_{t_0}^{t_0-\epsilon}\int_\Omega u^2dxdt + \frac{c}{\epsilon}\int_{t_1-\epsilon }^{t_1}\int_\Omega u^2dxdt .
\]
Therefore
\begin{equation}\label{34s}
\limsup_{\epsilon \rightarrow 0}\mathscr{J}_\epsilon ^2 \le \int_\Omega u^2(x ,t_0)dx+\int_\Omega u^2(x,t_1)dx.
\end{equation}

We rewrite \eqref{33s}  in the form
\begin{equation}\label{33s+}
C\|u\|_{L^2((t_0+\epsilon ,t_1-\epsilon ),H^1(\Omega ))}\le \|u\|_{L^2((t_0,t_1),L^2(\Gamma ))}+\|\nabla u\|_{L^2((t_0,t_1),L^2(\Gamma ))}+\mathscr{J}_\epsilon .
\end{equation}
We apply Hardy's inequality in Corollary \ref{corollaryH}. We obtain
\[
\|u\|_{L^2((t_0 ,t_0+\epsilon),H^1(\Omega ))},\; \|u\|_{L^2((t_1-\epsilon ,t_1),H^1(\Omega ))}\le C\epsilon^s \|u\|_{H^1((t_0 ,t_1),H^1(\Omega ))}.
\]

This and \eqref{33s+} yield
\begin{align*}
C\|u\|_{L^2((t_0 ,t_1),H^1(\Omega ))}\le \|u\|_{L^2((t_0,t_1),L^2(\Gamma ))}&+\|\nabla u\|_{L^2((t_0,t_1),L^2(\Gamma ))}
\\
&+\epsilon^s \|u\|_{H^1((t_0 ,t_1),H^1(\Omega ))}+\mathscr{J}_\epsilon.
\end{align*}
Making $\epsilon \rightarrow 0$, we get by using \eqref{34s}
\begin{align*}
C\|u\|_{L^2((t_0 ,t_1),H^1(\Omega ))}\le \|u\|_{L^2((t_0,t_1),L^2(\Gamma ))}&+\|\nabla u\|_{L^2((t_0,t_1),L^2(\Gamma ))}
\\
&+\|u(\cdot ,t_0)\|_{L^2(\Omega )}+\|u(\cdot ,t_1)\|_{L^2(\Omega )}.
\end{align*}
We complete the proof by using the following inequality
\[
C\left(\|u(\cdot ,t_0)\|_{L^2(\Omega )}+\|u(\cdot ,t_1)\|_{L^2(\Omega )}\right)\le \|u\|_{H^1((t_0,t_1),L^2(\Gamma ))}+\|\nabla u\|_{L^2((t_0,t_1),L^2(\Gamma ))}.
\]
To prove this inequality we proceed similarly to the proof of observability inequalities for parabolic equation. First, if $s_0=(3t_0+t_1)/4$ and $s_1=(t_0+3t_1)/4$, we get as a straightforward consequence of the Carleman inequality in Theorem \ref{theorem1.1+},
\begin{equation}\label{34s.1}
C\|u\|_{L^2((s_0,s_1)\times \Omega )}\le \|u\|_{H^1((t_0,t_1),L^2(\Gamma ))}+\|\nabla u\|_{L^2((t_0,t_1),L^2(\Gamma ))}.
\end{equation}
Next pick $\psi \in C^\infty ([t_0,t_1])$ so that $0\le \psi \le 1$, $\psi =0$ in $[t_0,s_0]$ and $\psi =1$ in $[s_1,t_1]$. Then $v=\psi u$ is the solution of the IBVP
\[
\left\{
\begin{array}{lll}
-\mbox{div}(A\nabla v)+\partial_t v=\psi 'u \;\; &\mbox{in}\; Q,
\\
v=u &\mbox{on}\; \Sigma ,
\\
v(\cdot ,t_0)=0.
\end{array}
\right.
\]
Hence
\[
\int_Q A\nabla v\cdot \nabla vdxdt+\int_\Sigma v(A\nabla v\cdot \nu) +\frac{1}{2}\int_Q \partial _t v^2dxdt=\int_Q \psi 'u v dxdt.
\]
But
\[
\int_Q \partial _t v^2dxdt=\int_\Omega v^2(x,t_1)dx.
\]
That  is we have
\[
\int_Q A\nabla v\cdot \nabla vdxdt+\int_\Sigma vA\nabla v\cdot \nu +\frac{1}{2}\int_\Omega v^2(x,t_1)dx=\int_Q \psi 'uvdxdt.
\]
We deduce from this identity
\begin{align}
&C\left(\|v(\cdot ,t_1)\|_{L^2(\Omega )}^2+\|\nabla v\|_{L^2(Q)}^2\right) \label{34s.2}
\\
&\hskip 1cm\le \|u\|_{L^2((t_0,t_1),L^2(\Gamma ))}^2+\|\nabla u\|_{L^2((t_0,t_1),L^2(\Gamma ))}^2 +\|\psi 'u\|_{L^2(Q)}\|v\|_{L^2(Q)}.\nonumber
\end{align}
Noting that 
\[
w\rightarrow \left(\int_\Omega |\nabla w|^2dx +\int_\Gamma w^2(x)d\sigma(x)\right)^{\frac{1}{2}}
\]
defines an equivalent norm on $H^1(\Omega )$, we get
\[
\|v\|_{L^2(Q)}^2\le c_\Omega\left(\|\nabla v\|_{L^2(Q)}^2 +\|u\|_{L^2((t_0,t_1),L^2(\Gamma ))}^2\right).
\]
We obtain then from Young's inequality 
\[
\|\psi 'u\|_{L^2(Q)}\|v\|_{L^2(Q)}\le \frac{1}{2\epsilon}\|\psi 'u\|_{L^2(Q)}^2+\frac{c_\Omega\epsilon}{2}\|\nabla v\|_{L^2(Q)}^2 +\frac{c_\Omega\epsilon}{2}\|u\|_{L^2((t_0,t_1),L^2(\Gamma ))}^2.
\]
This inequality in \eqref{34s.2}, with $\epsilon$ sufficiently small, yields
\begin{align}
C\|u(\cdot ,t_1)\|_{L^2(\Omega )}&=C\|v(\cdot ,t_1)\|_{L^2(\Omega )}\label{34s.3}
\\
&\le \|u\|_{L^2((t_0,t_1),L^2(\Gamma ))}+\|\nabla u\|_{L^2((t_0,t_1),L^2(\Gamma ))}+ \|\psi 'u\|_{L^2(Q )}.\nonumber
\end{align}
Bearing in mind that $\mbox{supp}(\psi ')\subset [s_0,s_1]$, we deduce from \eqref{34s.1} 
\begin{align*}
\|\psi 'u\|_{L^2(Q )}&\le C\|u\|_{L^2(\Omega \times (s_0,s_1))}
\\
&\le \|u\|_{H^1((t_0,t_1),L^2(\Gamma ))}+\|\nabla u\|_{L^2((t_0,t_1),L^2(\Gamma ))}.
\end{align*}
This in \eqref{34s.3} gives
\[
C\|u(\cdot ,t_1)\|_{L^2(\Omega )}\le \|u\|_{H^1((t_0,t_1),L^2(\Gamma ))}+\|\nabla u\|_{L^2((t_0,t_1),L^2(\Gamma ))}.
\]
As the Carleman estimate in Theorem \ref{theorem1.1+} still holds for the backward parabolic equation $\mbox{div}(A\nabla u)+\partial _tu=0$, we have similarly 
\[
C\|u(\cdot ,t_0)\|_{L^2(\Omega )}\le \|u\|_{H^1((t_0,t_1),L^2(\Gamma ))}+\|\nabla u\|_{L^2((t_0,t_1),L^2(\Gamma ))}.
\]
The proof is then complete.
\end{proof}

If $u\in \mathscr{Y}(Q)$ satisfies $Lu=0$ in $Q$ then $\partial_tu\in \mathscr{X}(Q)$ and $L\partial_t u=0$ in $Q$.
Proposition \ref{proposition3.1} applied to both $u$ and $\partial_tu$ together with Lemma \ref{lemma3.1} give the following result.

\begin{corollary}\label{corollary3.1}
Let $s\in (0,1/2)$. There exist $\omega \Subset \Omega$, only depending on $\Omega$, and three constants $c>0$, $C>0$ and $\sigma _0>0$, only depending on $\Omega$, $\kappa$,  $T_0$, $s$ and $\alpha$, so that, for any $u\in \mathscr{Y}(Q)$ satisfying $Lu=0$ in $Q$ and $0<\sigma <\sigma _0$, we have
\begin{equation}\label{3.3}
\|u\|_{L^2((t_0 ,t_1),H^1(\Omega ))}\le C\left(\sigma ^{\min\{\alpha ,s\}/2} \|u\|_{\mathscr{Y}(Q)} +e^{e^{c/\sqrt{\sigma}}}\|u\|_{H^1((t_0,t_1),H^1(\omega ))}\right).
\end{equation}
\end{corollary}

We now quantify the uniqueness of continuation from an interior subdomain to an another interior subdomain. Prior to that, we define the geometric distance $d_g^D$ on a bounded domain $D$ of $\mathbb{R}^n$ by
\[
d_g^D(x,y)=\inf\left \{ \ell (\psi ) ;\; \psi :[0,1]\rightarrow D \; \mbox{Lipschitz path joining}\; x \; \mbox{to}\; y\right\},
\] 
where
\[
\ell (\psi )= \int_0^1\left|\dot{\psi}(t)\right|dt
\]
is the length of $\psi$.

Observe that, according to Rademacher's theorem, any Lipschitz continuous function $\psi :[0,1]\rightarrow D$ is almost everywhere differentiable with $\left|\dot{\psi}(t)\right|\le k$ a.e. $t\in [0,1]$, where $k$ is the Lipschitz constant of $\psi$ (see for instance \cite[Theorem 4, page 279]{Ev}). In particular, $\ell (\psi )$ is well defined.

The following lemma will be used to prove the next proposition. We provide its proof  in Appendix \ref{A}.

\begin{lemma}\label{Glemma}
Let $D$ be a bounded Lipschitz domain of $\mathbb{R}^n$. Then $d_g^D\in L^\infty (D \times D )$.
\end{lemma}

\begin{proposition}\label{proposition3.2}
Let $\omega\Subset \Omega$, $\tilde{\omega}\Subset \Omega$ and $s\in \left(0,1/2\right)$. There exist three constants $\gamma>0$, $C>0$ and $\epsilon_0 >0$, only depending on $\Omega$, $\kappa$, $T_0$, $s$, $\omega$ and $\tilde{\omega}$, so that, for any $u\in H^2((t_0,t_1),H^1(\Omega ))$ satisfying $Lu=0$ in $Q$ and $0<\epsilon <\epsilon_0 $, we have
\begin{equation}\label{3.4}
C\|u\|_{H^1((t_0,t_1),H^1(\tilde{\omega}))}\le \epsilon ^s\|u\|_{H^2((t_0,t_1),H^1(\Omega ))}+e^{\gamma /\epsilon}\|u\|_{H^1((t_0,t_1),H^1(\omega))}.
\end{equation}
\end{proposition}

\begin{proof}
Pick $\Omega_0$ a Lipschitz domain so that $\Omega_0\Subset \Omega$, $\omega\Subset \Omega_0$ and $\tilde{\omega}\Subset \Omega_0$. 
Set then $d_0=\mbox{dist}(\overline{\Omega_0},\Gamma )$. Fix $0<\epsilon <\epsilon_0:=\min \left(d_0^2/9,1\right)$ and let $0<\delta <\sqrt{\epsilon}$.
Let $x_0\in \omega$, $x\in \overline{\tilde{\omega}}$ and let $\psi :[0,1]\rightarrow \Omega_0$ be a Lipschitz path joining $x_0$ to $x$ so that $\ell (\psi)\le d_g^{\Omega_0}(x_0,x)+1$. For simplicity's sake, we use in this proof the notation
\[
\mathbf{d}=\|d_g^{\Omega_0}\|_{L^\infty (\Omega_0\times \Omega_0)}.
\]

Let $\tau_0=0$ and $\tau_{k+1}=\inf \{\tau\in [\tau_k,1];\; \psi (\tau)\not\in B(\psi (\tau_k),\delta )\}$, $k\geq 0$. We claim that there exists an integer $N\geq 1$ so that $\psi (1)\in B\left(\psi(\tau_N),\delta/2 \right)$. If not, we would have $\psi (1)\not\in B\left(\psi (\tau_k),\delta/2 \right)$, for any $k\ge 0$. As the sequence $(\tau_k)$ is non decreasing and bounded from above by $1$, it converges to $\hat{\tau}\le 1$. In particular, there exists an integer $k_0\geq 1$ so that $\psi (t_k)\in B\left(\psi (\hat{\tau}),\delta/2\right)$, $k\ge k_0$. But this contradicts the fact that $\left|\psi (\tau_{k+1})-\psi (\tau_k)\right| =\delta$, for any $k\ge 0$.

Let us check that $N\le N_0$, where $N_0$ depends only on $\mathbf{d}$ and $\delta$. Pick $1\le j\le n$ so that 
\[
\max_{1\leq i\leq n} \left|\psi _i(\tau_{k+1})-\psi _i(\tau_k)\right| =\left|\psi _j(\tau_{k+1})-\psi _j(\tau_k)\right|.
\]
Then
\[
\delta \le n\left|\psi _j (\tau_{k+1})-\psi _j(\tau_k)\right|=n\left| \int_{\tau_k}^{\tau_{k+1}}\dot{\psi}_j(t)dt\right|\le  n\int_{\tau_k}^{\tau_{k+1}}\left|\dot{\psi}(t)\right|dt  .
\]
In consequence, where $\tau_{N+1}=1$, 
\[
(N+1)\delta \le n\sum_{k=0}^N\int_{\tau_k}^{\tau_{k+1}}\left|\dot{\psi}(\tau)\right|d\tau =n\ell (\psi)\le n(\mathbf{d}+1).
\]
Therefore
\[
N\le N_0=\left[ \frac{n(\mathbf{d}+1)}{\delta}\right].
\]
Here $[n(\mathbf{d}+1)/\delta ]$ is the integer part of $n(\mathbf{d}+1)/\delta$.

Let $x_k=\psi (t_k)$, $0\le k\le N$.  If $|z-x_{k+1}|<\delta$ then \[ |z-x_k|\le |z-x_{k+1}|+|x_{k+1}-x_k|<2\delta.\]  In other words, $B(x_{k+1},\delta )\subset B(x_k,2\delta)$.

We remark that we have also $B(x_k,3\delta )\subset \Omega$, for each $k$.

Define
\[
I_k=(t_0+k\epsilon , t_1- k\epsilon )\quad  k\ge 0.
\]
If $t_0^k=t_0+k\epsilon$ and  $t_1^k=t_1- k\epsilon$  then $I_k=(t_0^k,t_1^k)$ and $I_{k+1}=(t_0^k+\epsilon ,t_1^k-\epsilon)$.

Let $u\in H^2((t_0,t_1),H^1(\Omega )$ satisfying $Lu=0$ in $Q$. In this proof we use the following temporary notation
\[
M=\|u\|_{H^1((t_0,t_1),H^1(\Omega ))} .
\]

Taking into account that $\delta <\sqrt{\epsilon}$, we have from the three-cylinder inequality \eqref{2.1}
\[
\|u\|_{L^2(I_{k+1},H^1(B(x_{k+1},\delta )))}\le C_0 \delta^{-3}M^{1-\vartheta}\|u\|_{L^2(I_k,H^1(B(x_k,\delta )))}^\vartheta ,
\]
the constants $C_0$ and $\vartheta$, $0<\vartheta <1$, only depend on $\Omega$, $\kappa$ and $T_0$.

Set $\Lambda _k=\|u\|_{L^2(I_k,H^1(B(x_k,\delta )))}$, $0\le k\le N$, and $\Lambda _{N+1}=\|u\|_{L^2\left(I_{N+1},H^1\left(B\left(x,\delta/2 \right)\right)\right)}$. We can then rewrite this inequality in the form
\begin{equation}\label{4.1}
\Lambda_{k+1}\le C_0\delta^{-3}M^{1-\vartheta}\Lambda_k^\vartheta .
\end{equation}

Let $\beta =\vartheta^{N+1}$. We get in a straightforward manner from \eqref{4.1}
\[
\Lambda_{N+1}\leq (C_0 \delta^{-3})^{\frac{1-\vartheta^{N+2}}{1-\vartheta}}M^{1-\beta}\Lambda_0^\beta .
\]
Substituting if necessary $C_0$ by $\max \{C_0,1\}$, we may assume that $C_0 \ge 1$. Then the last inequality gives
\[
\Lambda_{N+1}\leq C \delta^{\frac{-3}{1-\vartheta}}M^{1-\beta}\Lambda_0^\beta .
\]
From here and until the end of the proof, $C$ is a generic constant, depending only on $\Omega$, $\kappa$, $\omega$, $\tilde{\omega}$, $s$ and $T_0$.

Young's inequality then yields, for $\sigma >0$,
\begin{align*}
\Lambda_{N+1} &\le C\delta^{\frac{-3}{1-\vartheta}}((1-\beta )\sigma ^{\frac{\beta}{1-\beta}}M+\beta\sigma^{-1}\Lambda_0)
\\
&\le  C\delta^{\frac{-3}{1-\vartheta}}(\sigma ^{\frac{\beta}{1-\beta}}M+\sigma^{-1}\Lambda_0).
\end{align*}

If $\delta $ is sufficiently small $B(x_0,\delta )\subset \omega$. On the other hand, $\overline{\tilde{\omega}}$ can be recovered by $O(\delta ^{-n})$ balls of radius $\frac{\delta}{2}$. Whence, bearing in mind that 
\[
I_{N+1}\supset J=(t_0+(N_0+1)\epsilon ,t_1-(N_0+1)\epsilon ),
\]
we have
\begin{equation}\label{equat1}
\|u\|_{L^2(J,H^1(\tilde{\omega}))}\le \delta^{-\frac{3}{1-\vartheta}-n}\left(\sigma ^{\frac{\beta}{1-\beta}}M+\sigma^{-1}\|u\|_{L^2((t_0,t_1),H^1(\omega))}\right).
\end{equation}
 Then we take $\sigma$ in \eqref{equat1} in order to satisfy
\[
\delta^{-\frac{3}{1-\vartheta}-n}\sigma ^{\frac{\beta}{1-\beta}}=\delta .
\]
In that case
\[
\sigma ^{-1}\le \delta^{-\frac{m}{\beta}},
\]
with $m=(4-\vartheta)/(1-\vartheta)+n$.

But
\[
\beta = \vartheta^{N+1}=e^{-(N+1)|\ln \vartheta |}\ge \vartheta e^{-N_0|\ln \vartheta |}\ge \vartheta e^{-\frac{n|\ln \vartheta |(\mathbf{d}+1)}{\delta }}.
\]
Hence
\[
\delta^{-\frac{3}{1-\vartheta}-n}\sigma ^{-1}\le e^{\gamma /\delta}.
\]
This inequality in \eqref{equat1} yields
\begin{equation}\label{equat1.1.0}
C\|u\|_{L^2(J,H^1(\tilde{\omega}))}\le \delta M+e^{\gamma/\delta}\|u\|_{L^2((t_0,t_1),H^1(\omega))}.
\end{equation}
Let $\aleph= 2n(\mathbf{d}+1)+\sqrt{\epsilon_0}$. Substituting $\epsilon_0$ by $\min \{\epsilon _0,(t_1-t_0)/(3\aleph)\}$, we may assume that $J_0=(t_0-\aleph \epsilon ,t_1+\aleph \epsilon)\ne \emptyset$. Taking $\delta =\sqrt{\epsilon}/2$, we get in straightforward manner that  $J\supset J_0$. Hence \eqref{equat1.1.0} yields
\begin{equation}\label{equat1.1}
C\|u\|_{L^2(J_0,H^1(\tilde{\omega}))}\le \sqrt{\epsilon} M+e^{2\gamma/\sqrt{\epsilon}}\|u\|_{L^2((t_0,t_1),H^1(\omega))}.
\end{equation}
We get by applying again Hardy's inequality in Lemma \ref{lemmaH}, for some fixed $s\in ( 0,1/2)$, 
\begin{equation}\label{equat2}
\|u\|_{L^2((t_0 ,t_0+\aleph\epsilon ),H^1(\tilde{\omega}))}\le c\aleph^s\epsilon ^sM,\quad \|u\|_{L^2((t_1-\aleph\epsilon ,t_1 ),H^1(\tilde{\omega}))}\le c\aleph^s\epsilon ^sM.
\end{equation}

Then \eqref{equat1.1} and \eqref{equat2} give
\[
C\|u\|_{L^2((t_0,t_1),H^1(\tilde{\omega}))}\le \epsilon ^{s/2}M+e^{2\gamma/\sqrt{\epsilon}}\|u\|_{L^2((t_0,t_1),H^1(\omega))}.
\]

Substituting  $\epsilon$ by $\epsilon^2$ and $2\gamma$ by $\gamma$, we obtain
\begin{equation}\label{+1}
C\|u\|_{L^2((t_0,t_1),H^1(\tilde{\omega}))}\le \epsilon ^s\|u\|_{H^1((t_0,t_1),H^1(\Omega ))}+e^{\gamma /\epsilon}\|u\|_{L^2((t_0,t_1),H^1(\omega))}.
\end{equation}
As $\partial_tu\in H^1((t_0,t_1),H^1(\Omega )$ satisfies $L\partial_tu=0$ in $Q$, \eqref{+1} is applicable with $u$ substituted by $\partial_tu$. That is we have
\begin{equation}\label{+2}
C\|\partial_tu\|_{L^2((t_0,t_1),H^1(\tilde{\omega}))}\le \epsilon ^s\|\partial_tu\|_{H^1((t_0,t_1),H^1(\Omega ))}+e^{\gamma /\epsilon}\|\partial_tu\|_{L^2((t_0,t_1),H^1(\omega))}.
\end{equation}
Putting together \eqref{+1} and \eqref{+2} to obtain the expected inequality.
\end{proof}

We are now ready to prove the result  quantifying the uniqueness of continuation from an interior data. Prior to do that, we need to introduce a definition. Set $\varrho_\ast =e^{-e}$ and, for $\mu >0$ and $\varrho_0\le \varrho_\ast$,
\[
\Phi_{\varrho_0,\mu}(\varrho )=
\left\{
\begin{array}{ll}
0 &\mbox{if}\; \varrho =0,
\\
(\ln \ln  |\ln \varrho |)^{-\mu} \quad &\mbox{if}\; 0<\varrho\le \varrho_0 ,
\\
\varrho &\mbox{if}\; \varrho \ge \varrho_0.
\end{array}
\right.
\]

\begin{theorem}\label{theorem3.1}
Let $\omega \Subset \Omega$ and $s\in (0,1/2)$. There exist two constants $C>0$ and $0<\varrho_0\le \varrho_\ast$, only depending on $\Omega$, $\kappa$, $\omega$, $\alpha$, $s$ and $T_0$, so that, for any $0\ne u\in \mathscr{Y}(Q)$ satisfying $Lu=0$ in $Q$, we have
\[
C\|u\|_{L^2((t_0 ,t_1),H^1(\Omega ))}\le \|u\|_{\mathscr{Y}(Q)}\Phi_{\varrho_0,\mu}\left(\frac{\mathcal{I}(u,\omega )}{\|u\|_{\mathscr{Y}(Q)}}  \right).
\]
Here $\mu =\min\{s,\alpha\}/4$ and $\mathcal{I}(u,\omega )=\|u\|_{H^1((t_0,t_1),H^1(\omega ))}$.
\end{theorem}

\begin{proof}
 From Corollary \ref{corollary3.1},
there exist $\tilde{\omega} \Subset \Omega$ and three constants $c>0$, $C>0$ and $\sigma_0>0$ so that, for any $u\in \mathscr{Y}(Q)$ satisfying $Lu=0$ in $Q$ and $0<\sigma <\sigma_0$, we have
\[
C\|u\|_{L^2((t_0 ,t_1),H^1(\Omega ))}\le \sigma ^\mu \|u\|_{\mathscr{Y}(Q)} +e^{e^{c/\sqrt{\sigma}}}\|u\|_{H^1((t_0,t_1),H^1(\tilde{\omega} ))}.
\]
Here $\mu=\min \{s,\alpha \}/2$.

But according to Proposition \ref{proposition3.2}, there exist three constants $\gamma>0$, $C>0$ and $\epsilon_0 >0$ so that, for any $u\in \mathscr{Y}(Q)$ satisfying $Lu=0$ in $Q$ and $0<\epsilon <\epsilon_0$, we have
\[
C\|u\|_{H^1((t_0,t_1),H^1(\tilde{\omega}))}\le \epsilon ^s\|u\|_{\mathscr{Y}(Q)}+e^{\gamma/\epsilon}\|u\|_{H^1((t_0,t_1),H^1(\omega))}.
\]

The last two inequalities yield
\begin{align}
C\|u\|_{L^2((t_0 ,t_1),H^1(\Omega ))}\le  (\sigma ^\mu &+\epsilon^s e^{e^{c/\sqrt{\sigma}}}) \|u\|_{\mathscr{Y}(Q)} \label{E1}
\\
&+e^{\gamma/\epsilon}e^{e^{c/\sqrt{\sigma}}}\|u\|_{H^1((0,T),H^1(\omega ))},\nonumber
\end{align}
for any $u\in \mathscr{Y}(Q)$ satisfying $Lu=0$ in $Q$, $0<\sigma <\sigma_0$ and $0<\epsilon <\epsilon_0$.

We assume, by reducing $\sigma _0$ if necessary, that $\sigma_0^\mu e^{-e^{c/\sqrt{\sigma_0}}}< \epsilon _0^s$. We get, by taking $\epsilon$ so that $\epsilon^s=\sigma^\mu e^{-e^{c/\sqrt{\sigma}}}$,
\begin{equation}\label{E2}
C\|u\|_{L^2((t_0 ,t_1),H^1(\Omega ))}\le  \sigma ^\mu  \|u\|_{\mathscr{Y}(Q)} +e^{e^{e^{c/\sqrt{\sigma}}}}\|u\|_{H^1((t_0 ,t_1),H^1(\omega ))},
\end{equation}
for any $u\in \mathscr{Y}(Q)$ satisfying $Lu=0$ in $Q$, and  $0<\sigma <\sigma_0$. 

Fix $u\in \mathscr{Y}(Q)$, non identically equal to zero, satisfying $Lu=0$ in $Q$. Let
\[
M=\frac{\|u\|_{L^2((t_0 ,t_1),H^1(\Omega ))}}{\|u\|_{\mathscr{Y}(Q)}}\quad \mbox{and}\quad N=\frac{\|u\|_{H^1((t_0 ,t_1),H^1(\omega ))}}{\|u\|_{\mathscr{Y}(Q)}}.
\]
Then \eqref{E2} can be rewritten as
\begin{equation}\label{E2.0}
CM\le  \sigma ^\mu   +e^{e^{e^{c/\sqrt{\sigma}}}}N,\quad 0<\sigma <\sigma_0.
\end{equation}
Define the function $\ell$ by $\ell (\sigma)=\sigma ^\mu e^{-e^{e^{c/\sqrt{\sigma}}}}$. 
If $N<\min \{\ell (\sigma_0) , \varrho_\ast\}=\varrho_0$ then there exists $\overline{\sigma}$ so that $\ell (\overline{\sigma})=N$. Changing $c$ if necessary, we have
\[
\frac{1}{N}\le e^{e^{e^{c/\sqrt{\overline{\sigma}}}}}.
\]
Or equivalently
\[
\overline{\sigma}\le [\ln \ln |\ln N|]^{-1/2}
\]
It follows readily by taking $\sigma=\overline{\sigma}$ in \eqref{E2.0} that
\begin{equation}\label{E2.1}
CM\le [\ln \ln |\ln N|]^{-\mu/2}.
\end{equation}
If $N\ge \varrho_0$ then obviously we have
\begin{equation}\label{E2.2}
M\le 1\le \frac{N}{\varrho_0}.
\end{equation}
The expected inequality follows then from \eqref{E2.1} and \eqref{E2.2}.
\end{proof}

\section{Stability of parabolic Cauchy problems}\label{section4}

An additional step is necessary to prove our stability estimate for the Cauchy problem. It consists in quantifying the uniqueness of continuation from the Cauchy data to an interior subdomain. 

\begin{proposition}\label{propositionS4}
 Let $\nu\in (0,1/2)$. There exist $\omega \Subset \Omega$, only depending on $\Omega$ and $\Gamma_0$, and two constants $C>0$ and $c >0$, only depending on $\Omega$, $\kappa$, $\Gamma _0$, $\nu$ and $T_0$, so that, for any $u\in H^2((t_0,t_1),H^2(\Omega ))$ satisfying $Lu=0$ in $Q$ and $0<\epsilon <(t_1-t_0)/2$, we have
 \begin{align}
C&\|u\|_{H^1 ((t_0 ,t_1),L^2(\omega))}\le \epsilon^\nu \|u\|_{H^2((t_0,t_1),H^1(\Omega ))}\label{S4.1}
\\
&\hskip 1.5 cm +e^{c/\epsilon^2}\left(\|u\|_{H^2((t_0,t_1),L^2(\Gamma _0))}+\|\nabla u\|_{H^1((t_0,t_1), L^2(\Gamma _0))}\right).\nonumber
\end{align}
\end{proposition}

\begin{proof}
Pick $0<\epsilon <(t_1-t_0)/2$, $0<\eta <\epsilon$ and let $s\in [t_0+\epsilon , t_1-\epsilon]$. Let $\tilde{x}\in \Gamma_0$ be arbitrarily fixed and let $R>0$ so that $B(\tilde{x},R)\cap \Gamma \subset \Gamma_0$. Take $x_0$ in the interior of $\mathbb{R}^n\setminus\overline{\Omega}$ sufficiently close to $\tilde{x}$ is such a way that $\rho=\mbox{dist}(x_0,K)<R$, where $K=\overline{B(\tilde{x},R)}\cap \Gamma_0$ (think to the fact that $\Omega$ is on one side of its boundary). Fix then $r>0$ in order to satisfy $B(x_0,\rho +r)\cap \Gamma \subset \Gamma_0$ and $B(x_0,\rho+\theta r)\cap \Omega\ne \emptyset$, for some $0<\theta <1$.

Let $\phi \in C_0^\infty (B(\tilde{x},\rho +r))$ satisfying $\phi=1$ on $B(\tilde{x},\rho +(\theta+1)r/2)$. Set, where $0<\delta <1$ is a constant to be specified in the sequel,
\begin{align*}
&Q_0=[B(x_0,\rho +r)\cap \Omega ]\times (s-\eta , s+\eta ),
\\
&Q_1=[B(x_0,\rho +\theta r)\cap \Omega]\times (s-\delta \eta , s+\delta \eta ),
\\
&Q_2=\{ [B(x_0,\rho+r)\setminus B(x_0,\rho+(\theta +1)r/2)]\cap \Omega \}\times (s-\eta ,s+\eta ),
\\
&\Sigma _0=[B(x_0,\rho+r)\cap \Gamma ]\times (s-\eta ,s+\eta ).
\end{align*}

We apply Theorem \ref{theorem1.1+}, with $Q$ substituted by $Q_0$, $\psi =(\rho +r)^2-|x-x_0|^2$, $g(t)=1/[(t-s+\eta)(s+\eta -t)]$  and $\lambda$ fixed, to $\phi u$ so that $u\in H^1((t_0,t_1),H^2(\Omega ))$ satisfies $Lu=0$ in $Q$ in order to obtain
\begin{align}
C\int_{Q_1}u^2e^{-2\tau \varphi}dxdt &\le \int_{Q_0}(L(\phi u))^2e^{-2\tau \varphi}dxdt \label{5.1}
\\
&\quad +\int_{\Sigma _0}(u^2+|\nabla u|^2+(\partial _tu)^2)e^{-2\tau \varphi}dxdt. \nonumber
\end{align}
Here and henceforth, $C$ is a generic constant only depending on $\Omega$, $\kappa$, $\nu$, $\Gamma_0$ and $T_0$.

But
\[
L(\phi u)=L\phi u+2A\nabla \phi \cdot \nabla u.
\]
Whence $\mbox{supp}(L(\phi u))\cap Q_0\subset Q_2$ together with \eqref{5.1} yield
\begin{align}
C\int_{Q_1}u^2e^{-2\tau \varphi}dxdt&\le \int_{Q_2}(u^2+|\nabla u|^2)e^{-2\tau \varphi}dxdt \label{5.2}
\\
&\quad +\int_{\Sigma _0}(u^2+|\nabla u|^2+(\partial _tu)^2)e^{-2\tau \varphi}dxdt.\nonumber
\end{align}

Let
\begin{align*}
&\alpha =\eta ^{-2}\left[e^{4\lambda (\rho +r)^2}-e^{\lambda (2(\rho +r) ^2-(\rho +\theta r)^2)}\right] :=\eta ^{-2}\tilde{\alpha},
\\
&\beta =\eta ^{-2}\left[e^{4\lambda (\rho +r)^2}-e^{2\lambda (\rho +r)^2}\right] :=\eta ^{-2}\tilde{\beta},
\\
&\gamma =\eta ^{-2}\left[e^{4\lambda (\rho +r)^2}-e^{\lambda (2(\rho +r) ^2-(\rho +(\theta +1)r/2)^2)}\right] :=\eta ^{-2}\tilde{\gamma}.
\end{align*}
Then it is straightforward to check that
\begin{align*}
&\varphi (x,t)\le \frac{\alpha}{1-\delta}\quad \mbox{in}\; Q_1,
\\
&\varphi (x,t)\ge \beta \quad \mbox{in}\; \Sigma _0,
\\
&\varphi (x,t)\ge \gamma \quad \mbox{in}\; Q_2 .
\end{align*}
Noting that 
\[
\frac{\beta }{\alpha}=\frac{\tilde{\beta}}{\tilde{\alpha}}<1<\frac{\gamma}{\alpha}=\frac{\tilde{\gamma}}{\tilde{\alpha}}, 
\]
we can choose $0<\varkappa <1$ so that
\[
\frac{1}{1-\delta}:=\varkappa \frac{\tilde{\beta}}{\tilde{\alpha}}+(1-\varkappa )\frac{\tilde{\gamma}}{\tilde{\alpha}}>1.
\]
With this choice of $\delta$, \eqref{5.2} yields
\begin{align}
C\int_{Q_1}u^2dxdt&\le e^{-4b\eta^{-2}\tau }\int_{Q_2}(u^2+|\nabla u|^2)dxdt \label{5.3}
\\
&\quad +e^{4a\eta^{-2}\tau }\int_{\Sigma _0}(u^2+|\nabla u|^2+(\partial _tu)^2)dxdt.\nonumber
\end{align}
Here $a=(1-\varkappa )(\tilde{\gamma} -\tilde{\beta} )/2$ and $b=\varkappa (\tilde{\gamma} -\tilde{\beta})/2$. 

Let $\eta =\epsilon /2$, $s_0=t_0+\epsilon -\delta\epsilon /2$, $s_1=s_0+\delta \epsilon/2\ldots s_k=s_0+k\delta \epsilon/2$.
Let $K=K(\epsilon )$ so that 
\[
\bigcup_{k=0}^K\left(s_k-\delta \epsilon/2,s_k+\delta \epsilon /2\right)\supset [t_0+\epsilon ,t_1-\epsilon].
\]
If $Q_j^k$ (resp. $\Sigma _0^k$) denotes $Q_j$ (resp. $\Sigma_0$), $j=1,2$, when $s$ is substituted by $s_k$, then it follows from \eqref{5.3}
\begin{align*}
C\sum_{k=0}^K\int_{Q_1^k}u^2dxdt&\le e^{a\epsilon^{-2}\tau }\sum_{k=0}^K\int_{Q_2^k}(u^2+|\nabla u|^2)dxdt
\\
&\quad +e^{-b\epsilon^{-2}\tau }\sum_{k=0}^K\int_{\Sigma _0^k}(u^2+|\nabla u|^2+(\partial _tu)^2)dxdt.
\end{align*}
Note that the intervals $\left(s_k-\frac{\epsilon}{2},s_k+\frac{\epsilon}{2}\right)$ overlap, but their union can cover at most two times a subdomain of $(t_0,t_1)$. Whence
\begin{equation}\label{EQ1}
CI\le e^{a\epsilon^{-2}\tau }N+e^{-b\epsilon^{-2}\tau }M,\quad \tau \ge \tau_0,
\end{equation}
where we used the  temporary notations
\begin{align*}
&I=\|u\|_{L^2([B(x_0,\rho+\theta r)\cap \Omega )]\times (t_0+\epsilon  ,t_1-\epsilon ))},
\\
&M=\|u\|_{L^2((t_0,t_1),H^1(\Omega ))},
\\
&N=\|u\|_{H^1((t_0,t_1),L^2(\Gamma _0))}+\|\nabla u\|_{L^2((t_0,t_1), L^2(\Gamma _0))}.
\end{align*}
In \eqref{EQ1}, we get by substituting $\tau$ by $\epsilon^2 \tau$
\begin{equation}\label{EQ2}
CI\le e^{a\tau }M+e^{-b\tau }N,\quad \tau \ge \tau_0/\epsilon ^2.
\end{equation}
Set
\[
\tau_1=\frac{\ln (N/M)}{a+b}.
\]

If $\tau _1\ge \tau_0/\epsilon ^2$ then $\tau=\tau_1$ in \eqref{EQ2} yields
\begin{equation}\label{EQ3}
CI\le M^\vartheta N^{1-\vartheta},
\end{equation}
with $\vartheta=\frac{b}{a+b}$.

When $\tau _1<\tau_0/\epsilon ^2$, we have
\[
M<e^{(a+b)/\epsilon^{2}\tau_0}N.
\]
This inequality entails
\begin{equation}\label{EQ4}
I\le M=M^{\vartheta}M^{1-\vartheta}\le M^{\vartheta}e^{(1-\vartheta)(a+b)\tau_0/\epsilon ^2}N^{1-\vartheta}.
\end{equation}
So, in any case, one of estimates \eqref{EQ3} and \eqref{EQ4} holds. In other words, we proved
\begin{align*}
e^{-c/\epsilon^2}&\|u\|_{L^2([B(x_0,\rho+\theta r)\cap \Omega )]\times (t_0+\epsilon  ,t_1-\epsilon ))}
\\
& \le \|u\|_{L^2((t_0,t_1),H^1(\Omega ))}^\vartheta\left(\|u\|_{H^1((t_0,t_1),L^2(\Gamma _0))}+\|\nabla u\|_{L^2((t_0,t_1), L^2(\Gamma _0))} \right)^{1-\vartheta}.
\end{align*}

Fix $\omega \Subset B(x_0,\rho+\theta r)\cap \Omega $. Then the last inequality implies
\begin{align*}
e^{-c/\epsilon^2}&\|u\|_{L^2(\omega\times (t_0+\epsilon  ,t_1-\epsilon ))}
\\
&\le \|u\|_{L^2((t_0,t_1),H^1(\Omega ))}^\vartheta\left(\|u\|_{H^1((t_0,t_1),L^2(\Gamma _0))}+\|\nabla u\|_{L^2((t_0,t_1), L^2(\Gamma _0))} \right)^{1-\vartheta}.
\end{align*}
Hence
\begin{align}
&C\|u\|_{L^2(\omega \times (t_0+\epsilon  ,t_1-\epsilon ))} \label{EQ5}
\\
&\quad \le \sigma ^{\gamma}e^{c/\epsilon^2}M+\sigma^{-1}e^{c/\epsilon^2}\left(\|u\|_{H^1((t_0,t_1),L^2(\Gamma _0))}+\|\nabla u\|_{L^2((t_0,t_1), L^2(\Gamma _0))}\right),\nonumber
\end{align}
for $\sigma >0$, where $\gamma =\frac{1-\vartheta}{\vartheta}$.

We get, once again from Hardy's inequality in Lemma \ref{lemmaH},
\[
\|u\|_{L^2(\omega \times (t_0  ,t_0+\epsilon ))}\le C\epsilon^{\nu}M_1(u),\quad \|u\|_{L^2(\omega \times (t_1-\epsilon  ,t_1 ))}\le C\epsilon^{\nu}M_1(u),
\]
where $M_1(u)=\|u\|_{H^1((t_0,t_1),H^1(\Omega ))}$.

Combined with \eqref{EQ5} this inequality yields
\begin{align*}
C&\|u\|_{L^2(\omega \times (t_0 ,t_1))}
\\
&\le \left(\sigma ^{\gamma}e^{c/\epsilon^2}+\epsilon^\nu\right)M_1+\sigma^{-1}e^{c/\epsilon^2}\left(\|u\|_{H^1((t_0,t_1),L^2(\Gamma _0))}+\|\nabla u\|_{L^2((t_0,t_1), L^2(\Gamma _0))}\right).
\end{align*}
In this inequality, we take $\sigma$ so that $\sigma ^\gamma=\epsilon^\nu e^{-c/\epsilon^2}$. Noting that $\sigma ^{-1}\le \epsilon^{-\nu/\gamma}$, we find
\begin{align}
C\|u\|_{L^2(\omega \times (t_0 ,t_1))}&\le \epsilon^\nu M_1(u)\label{-1}
\\
&+e^{c/\epsilon^2}\left(\|u\|_{H^1((t_0,t_1),L^2(\Gamma _0))}+\|\nabla u\|_{L^2((t_0,t_1), L^2(\Gamma _0))}\right).\nonumber
\end{align}
As we have seen in the preceding proof, inequality \eqref{-1} still holds when $u$ is substituted by $\partial_tu$. That is we have
\begin{align}
C\|\partial _tu\|_{L^2(\omega \times (t_0 ,t_1))}&\le \epsilon^\nu M_1(\partial_tu)\label{-2}
\\
&+e^{c/\epsilon^2}\left(\|u\|_{H^2((t_0,t_1),L^2(\Gamma _0))}+\|\nabla u\|_{H^1((t_0,t_1), L^2(\Gamma _0))}\right).\nonumber
\end{align}
We add side by side \eqref{-1} and \eqref{-2} in order to obtain the expected inequality.
\end{proof}

\begin{remark}
{\rm
In the preceding proof we assumed that $u\in H^2((t_0,t_1),H^2(\Omega ))$. But a quick inspection of the proof of this proposition shows that if fact the result can be extended to functions from $H^2((t_0,t_1),H^1(\Omega ))\cap H^1((t_0,t_1),H^2(\Omega ))$.
}
\end{remark}

Proposition \ref{propositionS4} gives only an estimate in $H^1((t_0 ,t_1),L^2(\omega ))$. But we can obtain from it an estimate in $H^1((t_0 ,t_1),H^1(\omega ))$ by using the following Caccioppoli type inequality for the parabolic equation $Lu=0$. 

\begin{lemma}\label{lemma5.1}
 Let $\omega _0\Subset \omega_1 \Subset \Omega$. There exists a constant $C>0$, only depending on $\Omega$, $\kappa$, $T_0$, $\omega_0$ and $\omega_1$, so that, for any $u\in H^2((t_0,t_1),H^2(\Omega ))$ satisfying $Lu=0$ in $Q$, we have
\begin{equation}\label{5.4+}
C\|u\|_{H^1((t_0 ,t_1),H^1(\omega _0))}\le \|u\|_{H^2((t_0,t_1),L^2(\omega _1))}.
\end{equation}
\end{lemma}

\begin{proof}
Let $u\in H^2((t_0,t_1),H^2(\Omega ))$ satisfying $Lu=0$ in $Q$. As we have done in the preceding proofs, it is sufficient to prove
\[
C\|u\|_{L^2((t_0 ,t_1),H^1(\omega _0))}\le \|u\|_{H^1((t_0,t_1),L^2(\omega _1))},
\]
because this inequality holds for both $u$ and $\partial_tu$.

By Green's formula, for any $v\in L^2((t_0 ,t_1),H_0^1(\Omega ))$, we have
\begin{equation}\label{5.5}
\int_{t_0}^{t_1}\int_\Omega A\nabla u\cdot \nabla vdxdt -\int_{t_0}^{t_1}\int_\Omega \partial _t uvdxdt=0.
\end{equation}
Let $\phi \in C_0^\infty (\omega_1 )$ satisfying $0\le \phi \le 1$ and $\phi =1$ in $\omega _0$.

Taking $v=\phi ^2  u$ in \eqref{5.5} we get in straightforward manner 
\[
\int_{t_0}^{t_1}\int_{\omega _1} \phi ^2A\nabla u\cdot \nabla udxdt=-2\int_{t_0}^{t_1}\int_{\omega_1} ( \phi \nabla u)\cdot (uA\nabla \phi)dxdt +\int_{t_0}^{t_1}\int_{\omega_1} \phi ^2\partial _tuudxdt.
\]

But
\[
\int_{t_0}^{t_1}\int_{\omega _1} \phi ^2A\nabla u\cdot \nabla vdxdt \ge \kappa \int_{t_0}^{t_1}\int_{\omega _1}\phi ^2|\nabla u|^2dxdt.
\]

Therefore 
\begin{align}
\kappa \int_{t_0}^{t_1}\int_{\omega _1}\phi ^2|\nabla u|^2dxdt \le -2\int_{t_0}^{t_1}\int_{\omega_1} &( \phi \nabla u)\cdot (uA\nabla \phi)dxdt\label{5.6}
\\
&+ \int_{t_0}^{t_1}\int_{\omega_1} \phi ^2\partial _tuudxdt. \nonumber
\end{align}

An elementary convexity inequality yields
\begin{align}
&2\left|\int_{t_0}^{t_1}\int_{\omega_1} (\phi \nabla u)\cdot (uA\nabla \phi)dxdt\right| \label{5.7}
\\
&\hskip 2cm\le \frac{\kappa}{2}\int_{t_0}^{t_1}\int_{\omega_1}\phi ^2|\nabla u|^2dxdt +C\int_{t_0}^{t_1}\int_{\omega_1 }u^2dxdt. \nonumber
\end{align}
On the other hand, we have
\begin{equation}\label{5.7.1}
\left| \int_{t_0}^{t_1}\int_{\omega_1} \phi ^2\partial _tuudxdt\right|\le \int_{t_0}^{t_1}\int_{\Omega} \phi ^2u^2dxdt +
\int_{t_0}^{t_1}\int_{\Omega} \phi^2 (\partial_tu)^2dxdt .
\end{equation}

Combining \eqref{5.6}, \eqref{5.7} and \eqref{5.7.1}, we end up getting
\[
C\|\nabla u\|_{L^2((t_0 ,t_1),L^2(\omega _0))}\le \|u\|_{L^2(\omega _1\times (t_0,t_1))}+\|\partial_tu\|_{L^2(\omega _1\times (t_0,t_1))}.
\]
Or equivalently 
\[
C\|u\|_{L^2((t_0 ,t_1),H^1(\omega _0))}\le \|u\|_{H^1((t_0,t_1),L^2(\omega _1))}
\]
as expected.
\end{proof}

An immediate consequence of Caccioppoli's  inequality \eqref{5.4+} and Proposition \ref{propositionS4} (applied both to $u$ and $\partial _tu$), we have

\begin{corollary}\label{corollaryS4}
Let $\nu \in (0,1/2)$. There exist $\omega \Subset \Omega$, only depending on $\Omega$ and $\Gamma_0$, and two constants $C>0$ and $c >0$, only depending on $\Omega$, $\kappa$, $T_0$, $\nu$ and $\Gamma _0$, so that, for any $u\in H^3((t_0,t_1),H^2(\Omega ))$ satisfying $Lu=0$ and $0<\epsilon <(t_1-t_0)/2$, we have
 \begin{align}
C\|u\|_{H^1((t_0 ,t_1),H^1(\omega ))}&\le \epsilon^\nu \|u\|_{H^3((t_0,t_1),H^1(\Omega ))}\label{S4.2}
\\
&+e^{c/\epsilon^2}\left(\|u\|_{H^3((t_0,t_1),L^2(\Gamma _0))}+\|\nabla u\|_{H^2((t_0,t_1), L^2(\Gamma _0))}\right).\nonumber
\end{align}
\end{corollary}

We are now in position to complete the proof of Theorem \ref{theorem1}. We recall that 
\[
\mathcal{C}(u,\Gamma _0)=\|u\|_{H^3((t_0,t_1),L^2(\Gamma _0))}+\|\nabla u\|_{H^2((t_0,t_1), L^2(\Gamma _0))}.
\]
If $M=\|u\|_{\mathscr{Z}(Q)}$ then, in light of inequality \eqref{E2} in the end of the proof of Theorem \ref{theorem3.1} and inequality \eqref{S4.2}, we get, for $0<\epsilon <(t_1-t_0)/2$ and $0<\sigma <\sigma_0$,
\[
C\|u\|_{L^2((t_0 ,t_1),H^1(\Omega ))}\le \left(\sigma ^{\min \{\nu ,\alpha\}/2} +e^{e^{e^{c/\sqrt{\sigma}}}}\epsilon ^\nu \right)M+e^{e^{e^{c/\sqrt{\sigma}}}}e^{c/\epsilon^2}\mathcal{C}(u,\Gamma _0),
\]
the constants $C>0$, $c >0$ and $\sigma_0>0$ only depend on $\Omega$, $\kappa$, $T_0$, $\nu$ and $\Gamma _0$. 

The rest of the proof in quite similar to that of Theorem \ref{theorem3.1}.

As we have noted above, we have $L\partial_t^ju=0$ in $Q$ as soon as $Lu=0$ in $Q$, for any integer $j\ge 0$. This observation allows us to state the following variant 
of Theorem \ref{theorem1}, where $\mathscr{Z}^j(Q)=\mathscr{Y}(Q)\cap H^{3+j}((t_0,t_1),H^2(\Omega ))$ is endowed with its natural norm
\[
\|u\|_{\mathscr{Z}^j(Q)}=\|u\|_{\mathscr{Y}(Q)}+\|u\|_{H^{3+j}((t_0,t_1),H^2(\Omega ))}
\]

\begin{theorem}\label{theorem1+}
Let $\Gamma_0$ be a nonempty open subset of $\Gamma$ and $s\in (0,1/2)$. Then there exist two constants $C>0$ and $0<\varrho_0\le \varrho^\ast$, depending on $\Omega$, $\kappa$, $T_0$, $\alpha$, $s$ and $\Gamma_0$,  so that, for any integer $j\ge 0$ and $u\in\mathscr{Z}^j(Q)$ satisfying $Lu=0$ in $Q$, we have
\[
C\|u\|_{H^j((t_0,t_1),H^1(\Omega ))}\le (j+1)\|u\|_{\mathscr{Z}^j(Q)}\Psi_{\varrho_0,\mu} \left(\frac{\mathcal{C}_j(u,\Gamma _0)}{\|u\|_{\mathscr{Z}^j(Q)}}\right),
\]
with $\mu =\min\{\alpha ,s)\}/4$ and
\[
\mathcal{C}_j(u,\Gamma _0)=\|u\|_{H^{j+3}((t_0,t_1),L^2(\Gamma _0))}+\|\nabla u\|_{H^{j+2}((t_0,t_1),L^2(\Gamma _0))}.
\]
\end{theorem}

\appendix

\section{}\label{A}
We are grateful to Tom ter Elst \cite{tE} for having communicated to us the proofs of Lemma \ref{Glemma} and Corollary \ref{corollaryA} bellow. 
We reproduce in this appendix these proofs.

\begin{proof}[Proof of Lemma \ref{Glemma}]  Let $D$ be a Lipschitz domain of $\mathbb{R}^n$ and introduce the notations
\begin{align*}
&Q=\{ x=(x',x_n)\in \mathbb{R}^n;\; |x'|<1,\; -1<x_n<1\},
\\
&Q_-=\{x\in E;\; x_n<0\},
\\
&Q_0=\{ x\in E;\; x_n=0\}.
\end{align*}
As $D$ is Lipschitz, if $\tilde{x}\in \Gamma$ then there exist a neighborhood $U$ of $\tilde{x}$ in $\mathbb{R}^n$ and a bijective map $\phi :Q \rightarrow U$  so that  $\phi : \overline{Q}\rightarrow \overline{U}$ and $\phi ^{-1}:\overline{U}\rightarrow \overline{Q}$ are Lipschitz continuous, and 
\[
\phi (Q_-)=\Omega \cap U\quad \phi (Q_0)=U\cap \Gamma.
\]

For $x,y\in D \cap U$ define
\[
\psi (t) = \phi (\phi^{-1}(x)+t[\phi^{-1}(y)-\phi^{-1}(x)]),\quad t\in [0,1].
\]
Clearly, noting that $Q_-$ is convex, $\psi$ is a Lipschitz path in $D$ joining $x$ to $y$. We have in addition, where $k_+$ and $k_-$ are the respective Lipschitz constants of $\psi$ and $\psi^{-1}$,
\[
|\psi (t)-\psi (s)|\le k_+|t-s||\phi^{-1}(y)-\phi^{-1}(x)|\le k_+k_-|t-s||x-y|,
\]
$t,s\in [0,1]$.

Therefore
\begin{equation}\label{A1}
\int_0^1\left|\dot{\psi}(t)\right|dt\le k|x-y|\le k\mbox{diam}(D).
\end{equation}
Here $k=k_+k_-$.

A compactness argument shows that $\overline{D}$ can be recovered by finite number of open subsets $U_j$ with $U_j$ is either a ball or an open subset of the form $U$. As $D$ is a domain then necessarily any $U_j$ intersect at least $U_\ell$ for some $\ell\ne j$. In consequence, any arbitrary two points $x,y\in D$ can be joined by a Lipschitz path consisting on finite number (independent on $x$ and $y$) of line segments and paths of the form $\psi$. Whence, in light of \eqref{A1}, we may find a constant $C>0$ depending only on $D$ so that $d_g^D(x,y)\le C$ for any $x,y\in D$.
\end{proof}

It is worth mentioning the following consequence of Lemma \ref{Glemma}. 
\begin{corollary}\label{corollaryA}
Let $D$ be a Lipschitz bounded domain of $\mathbb{R}^n$. Then there exists a constant $\varkappa >0$ so that
\[
d_g^D(x,y)\le \varkappa |x-y|\quad \mbox{for any}\; x,y\in D .
\]
\end{corollary}

\begin{proof}
We proceed by contradiction. Assume then that there is no $\varkappa >0$ so that $d_g^D(x,y)\le \varkappa |x-y|$, for any $x,y\in D$. In particular, for any positive integer $i$, we may find two sequences $(x_i)$ and $(y_i)$ in $D$ so that $x_i\ne y_i$ and $d_g^D(x_i,y_i)>i|x_i-y_i|$, for each $i$. Thus
\begin{equation}\label{A2}
|x_i-y_i|\le \frac{1}{i}\|d_g^D\|_{L^\infty (D\times D)},\quad \mbox{for each}\; i.
\end{equation}
Subtracting if necessary a subsequence,  we may assume that $x_i$ converges to some $x\in \overline{D}$. Using \eqref{A2} we see that $y_i$ converges also to $x$. Fix $j$ so that $x\in U_j$, where $U_j$ is as in the preceding proof. According to \eqref{A1}, we have
\[
d_g^D(y,z)\le K|y-z|,\quad \mbox{for any}\; y,z\in U_j,
\]
for some constant $K>0$.

On the other hand, there exists a positive integer $i_0$ so that $x_i,y_i\in U_j$, for $i\ge i_0$. Hence, for any $i\ge \max \{i_0,K\}$, we get
\[
i|x_i-y_i|\ge K|x_i-y_i|\ge d_g^D(x_i,y_i)>i|x_i-y_i|.
\]
This leads to the expected contradiction.
\end{proof}

\vskip .5cm
\end{document}